\documentclass[11pt,a4paper,reqno]{amsart}
\usepackage{color,latexsym,amsfonts,amssymb,comment,dsfont}
\usepackage{amsmath,cite}
\usepackage{amsthm}
\usepackage{fullpage}
      \usepackage{layout}
\usepackage{graphicx}
\usepackage{tikz}

\newcommand{\B}{\mathbb{B}}

\newcommand{\E}{\mathbb{E}} 
\newcommand{\Z}{\mathbb{Z}}
\newcommand{\R}{\mathbb{R}}

\newcommand{\SIR}{\mathsf{SIR}}

\DeclareMathOperator*{\esssup}{ess\,sup}

\def\eq{\begin{equation}}
\def\en{\end{equation}}

\newtheorem{Theorem}{Theorem}[section]
\newtheorem{Corollary}[Theorem]{Corollary}
\newtheorem{Lemma}[Theorem]{Lemma}
\newtheorem{Proposition}[Theorem]{Proposition}

\theoremstyle{definition}

\newtheorem{Remark}[Theorem]{Remark}

\def\rr{{\varrho}}
\def\e{{\varepsilon}}

\def\a{\alpha}

\def\b{\beta}

\def\e{\varepsilon}

\def\phi{\varphi}
\def\g{\gamma}
\def\la{\lambda}
\def\k{\kappa}
\def\r{\rho}
\def\de{\delta}

\def\s{\sigma}
\def\t{{\ttau}}

\def\L{\Lambda}
\def\G{\Gamma}

\def\P{{\Phi}}

\def\T{\T}

\def\BB{{\mc{B}}}
\def\BBB{{\mc{K}}}

\def\LL{{\mc L}}

\def\V|{{\Vert}}

\def\bb{{\bf{b}}}
\def\aa{{\bf{a}}}
\def\cc{{\bf{c}}}

\def\ttau{\boldsymbol{\tau}}
\def\nuu{{\bar\nu}}
\def\muu{{\bar\mu}}

\def\usc{\text{u.s.c.~}}
\def\lsc{\text{l.s.c.~}}

\def\d{{\rm d}}
\def\E{\mathbb{E}}
\def\I{\mathsf{I}}

\def\one{\mathds{1}}
\def\mc{\mathcal}

\def\ms{\mathsf}
\def\N{\mathbb{N}}
\def\P{\mathbb{P}}
\def\R{\mathbb{R}}
\def\Z{\mathbb{Z}}

\begin{document}
\author{Christian Hirsch, Benedikt Jahnel, Paul Keeler, Robert Patterson}
\thanks{Weierstrass Institute Berlin, Mohrenstr. 39, 10117 Berlin, Germany; E-mail: \{{\tt christian.hirsch}, {\tt benedikt.jahnel}, {\tt paul.keeler}, {\tt robert.patterson}\}{\tt @wias-berlin.de}.}

\title{Large deviations in relay-augmented wireless networks}

\date{\today}

\begin{abstract}
We analyze a model of relay-augmented cellular wireless networks. The network users, who move according to a general mobility model based on a Poisson point process of continuous trajectories in a bounded domain, try to communicate with a base station located at the origin. Messages can be sent either directly or indirectly by relaying over a second user. We show that in a scenario of an increasing number of users, the probability that an atypically high number of users experiences bad quality of service over a certain amount of time, decays at an exponential speed. This speed is characterized via a constrained entropy minimization problem.  
Further, we provide simulation results indicating that solutions of this problem are potentially non-unique due to symmetry breaking. Also two general sources for bad quality of service can be detected, which we refer to as isolation and screening.
\end{abstract}

\maketitle

\section{Model definition and main results}
\label{intSec}
In classical cellular networks users communicate directly with a base station over a wireless channel. This network paradigm has been the foundation of modern cellular telecommunication and, so far, electrical engineers have managed to adapt this model to new technological developments. However, as the growing number of user devices makes it increasingly difficult to provide adequate \textit{quality of service} (QoS) to all users within a certain cell, LTE-A is the first standard to allow for augmenting the classical cellular set-up by the concept of relays \cite{3GPPRelay}. That is, instead of communicating directly with a possibly distant base station, user devices can now connect to the base station indirectly by routing via a nearby relay. Hence, using relays allows for extension of the coverage area of the base station and for offloading traffic from direct connections. 

\medskip
In this paper, we investigate a probabilistic model for the effect of relaying in a single cell  in the asymptotic setting of a large number of mobile users. Note that this is different from the thermodynamic limit considered in~\cite{adHoc} where both, the number of users and the size of the domain, tend to infinity. For related work in various non-asymptotic settings, we refer the reader to~\cite{decrRel,m2mHop,sinrPerc,fogComp}. We assume the existence of a single base station located at the origin. 
The mobile users in the associated cell are given by a Poisson point process $X^\lambda$ of trajectories with intensity function $\lambda\muu(\cdot)$, where $\lambda>0$. We assume that the distribution of the initial points of trajectories is absolutely continuous with respect to the Lebesgue measure. 
Moreover, $\muu$ is assumed to be a finite Borel measure on  the set of Lipschitz-continuous trajectories $\LL=\LL_{J_1}({I},W)$, with Lipschitz parameter $J_1$, from the time interval $I=[0,T)$ to a window $W$. Here $\LL$ is equipped with the supremum norm and $W$ is of the form $W=[-r,r]^d$ for some integer $r\ge1$. For instance these conditions would be satisfied for a random-waypoint model with bounded velocities, as described in~\cite{adHoc}.

\medskip
For the network model, we follow the classical approach based on the \textit{signal-to-interference ratio} (SIR)~\cite{baccelli2009stochastic1}. To be more precise, we let $\ell:[0,\infty)\to(0,\infty)$ denote the \emph{path-loss function}, which is a Lipschitz-continuous function, with parameter $J_2$, that describes the decay of the signal strength over distance, hence it can be quite general as long as there is no singularity. Additionally, the ability of a receiver to decode a message is reduced by interference coming from other users. In the literature, it is often assumed that interference caused by relays can be neglected~\cite{relAss1,relAss2} since it is small when compared to the interference generated by actively transmitting users. In contrast, in this paper we consider a scenario where extensive relaying may occur and therefore we also take the relay-induced interference into account. Hence, our approach is related to the scenario considered in~\cite{relAss3}.
Moreover, let us mention, that in our model we do not consider any form of medium access control, which would attempt to reduce interference by coordinating the periods of active user transmissions. 
In other words, in our model, the interference is generated by \textit{all} users.

\medskip
To be more precise, at time $t\in{I}$ consider a fixed location $\eta\in W$. Then the interference at $\eta$ is given by 
$$\I(\eta,X^\lambda_t)=\sum_{X_{i}\in X^\lambda}\ell(|X_{i,t}-\eta|)$$
where $X_{i,t}$ denotes the $i$-th trajectory in $X^\la$ at time $t$.
Introducing the empirical measures
$$L_\lambda=\frac{1}{\lambda}\sum_{X_i\in X^{\lambda}}\delta_{X_i}\hspace{1cm}\text{ and }\hspace{1cm}L_{\la,t}=\frac{1}{\lambda}\sum_{X_i\in X^{\lambda}}\delta_{X_{i,t}}$$ 
respectively as a random element in $\mc{M}(\LL)$, the space of finite Borel measures on $\LL$, and in $\mc{M}(W)$, the space of finite Borel measures on $W$. We note that the interference $\I(\eta,X^\lambda)$ can be conveniently expressed as 
$$\I(\eta,X^\lambda_t)=\lambda L_{\la,t}(\ell(|\cdot-\eta|)).$$
Now, at time $t\in{I}$ we define the SIR of a transmitter at $\xi\in W$ and measured, at the same time, at a receiver position $\eta\in W$ as
$$\SIR_{\la}(\xi,\eta,L_{\lambda,t})=\frac{\ell(|\xi-\eta|)}{\la L_{\la,t}(\ell(|\cdot-\eta|))}.$$
Note that the denominator consists of the superposition of signal strengths coming from all network users, even if the transmission originates at some user $x=X_i\in X^\lambda$. This does not comply with the standard convention of omitting the signal of interest from the interference in the denominator. However, as we work in the limit where $\lambda$ tends to infinity, contributions from a finite number of users can be removed or added without influencing the final result. For the same reason, our model does not include noise.

\medskip
By Shannon's formula~\cite[Section 16]{baccelli2009stochastic2}, minimum data transmission rate requirements are equivalent to lower bounds on the SIR. That is, a connection between $\xi,\eta\in W$ is useful only if
$$\SIR_{\la}(\xi,\eta,L_{\lambda,t})=\frac{\ell(|\xi-\eta|)}{\la L_{\la,t}(\ell(|\cdot-\eta|))}\ge \r.$$
In particular, if $\r$ is of the form $\r=\lambda^{-1}\r'$, then the above requirement can be re-expressed as $\SIR(\xi,\eta,L_{\lambda,t})\ge \r'$, where $\SIR(\xi,\eta,L_{\lambda,t})=\la\SIR_{\la}(\xi,\eta,L_{\la,t})$. 

\medskip
This mathematical setting can model different types of telecommunication systems. First, it can be interpreted in the setting of machine-to-machine networks, where the number of devices is large but the amount of data in each transmission is small~\cite{m2mHop}. Hence, a comparatively small SIR threshold can be sufficient to transmit messages successfully. Second, our model can also be considered within a spread-spectrum setting with interference cancellation factor $\lambda^{-1}$. Hence, the limit $\lambda$ tending to infinity describes a scenario approaching perfect interference cancellation. We refer the reader to~\cite{sinrPerc,vazIyer,weberSIC,Keeler15} for further investigations of scenarios with substantial interference cancellation. 

In the following, we conduct level-2 large-deviation analysis of certain frustration events. In particular, we will see that the most likely option for a rare event to occur can be described by a certain finite Borel measure $\nu\in\mc{M}(W)$ that describes the asymptotic configuration of users under conditioning on the rare event.
Therefore, we extend the definition of SIR to arbitrary finite, positive Borel measures $\nu\in\mc{M}(W)$ and also write $$\SIR(\xi,\eta,\nu)=\frac{\ell(|\xi-\eta|)}{\nu(\ell(|\cdot-\eta|))}$$
for any $\xi,\eta\in W$.

\medskip
In order to keep the model flexible, we assume that the QoS of the direct link between $\xi$ and $\eta$ is given by
$$D(\xi,\eta,L_{\lambda,t})=g(\SIR(\xi,\eta,L_{\lambda,t})),$$
where $g:[0,\infty)\to[0,\infty)$ is a Lipschitz-continuous function which is strictly increasing on $[0,\r_+)$ and constant equal to $c_+$ on $[\r_+,\infty)$ for some $\r_+,c_+>0$.
As for the SIR we define $D$ for general $\nu\in\mc{M}(W)$ by $D(\xi,\eta,\nu)=g(\SIR(\xi,\eta,\nu))$. Moreover, we set $D(\xi,\eta,\nu)=c_+$ if $\nu(W)=0$.
For instance, possible choices of $g$ include $g(r)=\min\{r,K\}$ or, motivated by Shannon's capacity formula, $g(r)=\min\{\log(1+r),K\}$ for some fixed $K>0$.

\medskip
If, a message is sent out from a user $\xi$ to a user $\eta$ by routing via a relay at $\zeta$, then the quality of the relayed message transmission depends on both, the SIR from $\xi$ to $\zeta$ as well as on the SIR from $\zeta$ to $\eta$. We will assume that message transmissions are successful if the SIR of both links are above a certain threshold. In other words, we assume that the connection quality experienced when relaying from $\xi$ via $\zeta$ to $\eta$ can be expressed as 
$$\Gamma(\xi,\zeta,\eta,L_{\lambda,t})=\min\{D(\xi,\zeta,L_{\lambda,t}),D(\zeta,\eta,L_{\lambda,t})\}.$$
In Figure~\ref{modFig} we give a snap-shot illustration of this communication model using relays at time zero.
\begin{figure}[!htpb]
 \centering 
\begin{tikzpicture}[scale=2.0]
\fill[black] (3,3) circle (1.5pt);
\fill[red] (2.020,2.804) circle (0.8pt);
\fill[red] (3.296,4.495) circle (0.8pt);
\fill[red] (2.498,1.099) circle (0.8pt);
\fill[red] (1.052,1.185) circle (0.8pt);
\fill[red] (4.203,1.117) circle (0.8pt);
\fill[red] (0.803,3.531) circle (0.8pt);
\fill[red] (2.849,3.220) circle (0.8pt);
\fill[red] (2.261,3.491) circle (0.8pt);
\fill[red] (1.702,3.067) circle (0.8pt);
\fill[red] (0.993,4.293) circle (0.8pt);
\fill[red] (0.859,2.107) circle (0.8pt);
\fill[red] (2.279,2.432) circle (0.8pt);
\fill[red] (4.912,3.431) circle (0.8pt);
\fill[red] (4.418,1.599) circle (0.8pt);
\fill[red] (3.583,2.666) circle (0.8pt);
\fill[red] (1.752,4.332) circle (0.8pt);
\fill[red] (4.534,1.758) circle (0.8pt);
\fill[red] (1.574,0.987) circle (0.8pt);
\fill[red] (3.423,4.529) circle (0.8pt);
\fill[red] (3.502,1.996) circle (0.8pt);
\fill[red] (2.137,1.242) circle (0.8pt);
\fill[red] (4.253,2.455) circle (0.8pt);
\fill[red] (4.820,4.621) circle (0.8pt);
\fill[red] (1.043,4.836) circle (0.8pt);
\fill[red] (2.205,3.371) circle (0.8pt);
\fill[red] (3.699,3.074) circle (0.8pt);
\fill[red] (1.967,4.629) circle (0.8pt);
\fill[red] (4.802,0.662) circle (0.8pt);
\fill[red] (3.720,3.441) circle (0.8pt);
\fill[red] (5.485,4.327) circle (0.8pt);
\fill[red] (2.541,2.527) circle (0.8pt);
\fill[red] (0.857,5.155) circle (0.8pt);
\fill[red] (5.365,4.172) circle (0.8pt);
\fill[red] (2.083,3.176) circle (0.8pt);
\fill[red] (3.346,3.316) circle (0.8pt);
\draw[blue](2.020,2.804)-- (3,3);
\draw[blue](2.849,3.220)-- (3,3);
\draw[blue](2.261,3.491)-- (3,3);
\draw[blue](2.279,2.432)-- (3,3);
\draw[blue](3.583,2.666)-- (3,3);
\draw[blue](2.205,3.371)-- (3,3);
\draw[blue](3.699,3.074)-- (3,3);
\draw[blue](3.720,3.441)-- (3,3);
\draw[blue](2.541,2.527)-- (3,3);
\draw[blue](2.083,3.176)-- (3,3);
\draw[blue](3.346,3.316)-- (3,3);
\draw[dashed](3.296,4.495)-- (3.720,3.441);
\draw[dashed](3.296,4.495)-- (3.346,3.316);
\draw[dashed](1.702,3.067)-- (2.020,2.804);
\draw[dashed](1.702,3.067)-- (2.261,3.491);
\draw[dashed](1.702,3.067)-- (2.279,2.432);
\draw[dashed](1.702,3.067)-- (2.205,3.371);
\draw[dashed](1.702,3.067)-- (2.541,2.527);
\draw[dashed](1.702,3.067)-- (2.083,3.176);
\draw[dashed](4.912,3.431)-- (3.720,3.441);
\draw[dashed](1.752,4.332)-- (2.261,3.491);
\draw[dashed](1.752,4.332)-- (2.205,3.371);
\draw[dashed](3.423,4.529)-- (3.720,3.441);
\draw[dashed](3.502,1.996)-- (3.583,2.666);
\draw[dashed](3.502,1.996)-- (3.699,3.074);
\draw[dashed](3.502,1.996)-- (2.541,2.527);
\draw[dashed](4.253,2.455)-- (3.583,2.666);
\draw[dashed](4.253,2.455)-- (3.699,3.074);
\draw[dashed](4.253,2.455)-- (3.720,3.441);
\end{tikzpicture}
  \caption{Realization of the network model. Points connected to the origin via blue solid lines represent users with direct connection to the base station. Black dashed lines indicate possible connections of users that cannot directly communicate with the base station but can communicate with users with direct connection to the base station.}\label{modFig}  
\end{figure}
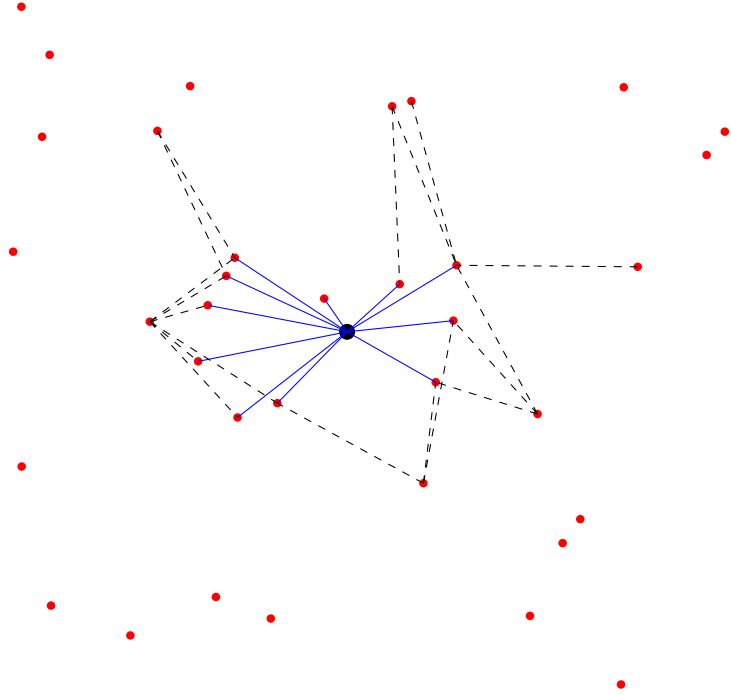

On the technical level of relays, the definition of $\G$ means that we consider full-duplex relaying. That is messages are sent and received over the same frequency channel. Although half-duplex relays are often used today, advances in techniques for canceling self-interference indicate that full-duplex relays will become an important component in fifth-generation networks~\cite{fullDuplex1}.

\medskip
In the following we introduce several characteristics that describe the QoS in a relay setting.

\subsection{Uplink and downlink quality of service}
In the uplink scenario, the destination of messages sent out from $X_i\in X^{\lambda}$ by routing via a relay $X_j\in X^{\lambda}$ is the origin $o$. 
Under an optimum relay decision, the QoS for the relayed uplink communication can be expressed as 
\begin{align}
\label{tnDef}
R(X_{i,t},o,L_{\lambda,t})=\max\{D(X_{i,t},o,L_{\lambda,t}),\max_{X_j\in X^{\lambda}}\Gamma(X_{i,t},X_{j,t},o,L_{\lambda,t})\}
\end{align}
In other words in \eqref{tnDef}, the user $X_i$ has the possibility to try to connect to the base station also directly. However, if there is any other user $X_j$ such that relaying via $X_j$ offers a better connection, then relaying leads to a higher QoS. A similar criterion has also been suggested in the engineering literature, see~\cite{bletsas1,bletsas2,relAss2}.

\medskip
Similarly, if a message is sent out from the origin $o$ to a user $X_i\in X^{\lambda}$, at time $t$, by routing via a relay $X_j\in X^{\lambda}$, then the quality of the relayed message transmission depends on both
$\SIR(o,X_{j,t},L_{\la,t})$ and $\SIR(X_{j,t},X_{i,t},L_{\la,t})$ via $\G(o,X_{j,t},X_{i,t},L_{\la,t})$.
Assuming an optimum relay decision, the QoS for the relayed downlink communication can then be expressed as 
\begin{align}
\label{tnDefDown}
R(o,X_{i,t},L_{\la,t})=\max\{D(o,X_{i,t},L_{\la,t}),\max_{X_{j}\in X^{\lambda}}\G(o,X_{j,t},X_{i,t},L_{\la,t})\}.
\end{align}

\medskip

We can further extend the definition of $R$ to arbitrary finite Borel measures $\nu\in\mc{M}(W)$ and write 
\begin{align*}
R(\xi,\eta,\nu)=\max\{D(\xi,\eta,\nu),\nu\text{-}\esssup_{\zeta\in W}\Gamma(\xi,\zeta,\eta,\nu)\}
\end{align*}
for any given $\xi,\eta\in W$. Here $\nu\text{-}\esssup$ denotes the essential supremum w.r.t.~$\nu$. 
Let $\pi_t: \LL\to W, \, x\mapsto x_t$ denote the projection at time $t\in{I}$, then for the trajectory of QoS, i.e.~for $\nuu\in\mc{M}(\LL)$ and $x,y\in\LL$ we define
\begin{align*}
\bar R(x,y,\nuu)=(R(x_t,y_t,\nuu_t))_{t\in I}
\end{align*}
where $\nuu_t=\nu\circ\pi_t^{-1}$ and similarly $\overline\SIR$ and $\bar D$. Note that $\bar R, \overline\SIR$ and $\bar D$ are elements of the space of bounded measurable functions $\BB=\BB(I,[0,\infty))$ equipped with the supremum norm and the associated Borel sigma field. 
We will show in Lemmas~\ref{PathRegularities} and \ref{PathRegularities2} that the path $t\mapsto R(x_t,o,\nuu_t)$ is continuous and also the map $x\mapsto\bar R(x,o,\nuu)$ is continuous.

\subsection{Statement of results}
The point processes of users that are frustrated due to failing to attain certain types of QoS is the principal object of investigation of this paper. 
For the uplink, the rescaled random measure associated with this point process is defined as
$$L_{\lambda}^{\ms{up}}[\tau]=\frac1\lambda\sum_{X_{j}\in X^{\lambda}}\delta_{X_{j}}\tau(\bar R(X_{j},o,L_{\la}))$$
where $\tau:\, \BB\to[0,\infty)$ is a bounded and measurable function.
In particular $L_{\lambda}^{\ms{up}}[\tau]\in\mc{M}(\LL)$.
More generally, if $\nuu\in\mc{M}(\LL)$ then $\nuu^{\ms{up}}[\tau]$ is defined as a measure in $\mc{M}(\LL)$ via 
$$\frac{\d\nuu^{\ms{up}}[\tau]}{\d\nuu}(x)=\tau(\bar R(x,o,\nuu)).$$

\medskip
Note that $\nuu^{\ms{up}}_t[\tau]$ does not suffice to describe traffic overflow arising from a large number of users communicating via a small number of relays. In practice, users communicating via the same relay have to share its bandwidth. Consequently, even if a user has good QoS but no direct connection to the base station, communication at full bandwidth cannot be guaranteed. In other words, the system can have many connected users but still suffer from small throughput. In that sense, also the random measure of users that have bad QoS, with respect to direct communication with the base station
$$L_{\lambda}^{\ms{up-dir}}[\tau]=\frac1\lambda\sum_{X_j\in X^{\lambda}}\delta_{X_j}\tau(\bar D(X_{j},o,L_{\la}))$$ 
is an important quantity. 
Again for $\nuu\in\mc{M}(\LL)$, $\nuu^{\ms{up-dir}}[\tau]$ is defined via
$$\frac{\d\nuu^{\ms{up-dir}}[\tau]}{\d\nuu}(x)=\tau(\bar D(x,o,\nuu)).$$
For the downlink we define 
$$\frac{\d\nuu^{\ms{do}}[\tau]}{\d\nuu}(x)=\tau(\bar R(o,x,\nuu)).$$
and analogously for $\nuu^{\ms{do-dir}}[\tau]$.
Since our main theorem will be about large deviations of all the four above quantities, let us introduce the following short hand notation
$$L_{\la}[\ttau]=\big(L_{\la}^{\ms{up}}[\tau_1],L_{\la}^{\ms{up-dir}}[\tau_2],L_{\la}^{\ms{do}}[\tau_3],L_{\la}^{\ms{do-dir}}[\tau_4]\big)$$
and, more generally,
\begin{align}
\label{nuTauDef}
\nuu[\ttau]=\big(\nuu^{\ms{up}}[\tau_1],\nuu^{\ms{up-dir}}[\tau_2],\nuu^{\ms{do}}[\tau_3],\nuu^{\ms{do-dir}}[\tau_4]\big)
\end{align}
where $\ttau=(\tau_i)_{i\in\{1,\dots,4\}}$.

\medskip
Let $\BBB=\BBB(I,W)$ denote the space of measurable trajectories with values in $W$, equipped with the supremum norm.
We are interested in random variables $F(L_\la[\ttau])$ where  $F:\, \mc{M}(\BBB)^4\to[-\infty,\infty)$ and $\tau_i:\, \BB\to[0,\infty)$, $i\in\{1,\dots,4\}$, exhibit some appropriate monotonicity properties. More precisely, 
$\tau_i$ is assumed to be a decreasing in the sense that for all $\g,\g'\in \BB$ with $\g_t\le\g'_t$ for all $t\in I$ we have $\tau(\g)\ge\tau(\g')$. Moreover, $F$ is assumed to be increasing in the sense that for all $\nuu,\nuu'\in\mc{M}(\BBB)$ 
with $\nuu\le\nuu'$ we have $F(\nuu)\le F(\nuu')$. Here we write $\nuu\le\nuu'$ if $\nuu(A)\le\nuu'(A)$ for all measurable $A\subset\BBB$. 
We also put $\nuu<\nuu'$ if $\nuu\le\nuu'$ and $\nuu\ne\nuu'$. 

\medskip
For example, consider the measurable functions $F_\bb:\mc{M}(\BBB)^4\to[-\infty,\infty)$,
\begin{align}
\label{ExF}
(\nuu_i)_{i\in\{1,\dots,4\}}\mapsto
\begin{cases}
0& \text{if } \nuu_i(\BBB)> b_i\text{ for all }i\in\{1,\dots,4\}\\
-\infty& \text{otherwise }
\end{cases}
\end{align}
for some $\bb\in\R^4$ 
and $\tau_{a,c}:\, \BB\to[0,\infty)$, 
\begin{align}
\label{ExF2}
\g\mapsto
\begin{cases}
1& \text{if } \int_0^T\one\{\g_t<c\}\d t> a,\\
0& \text{otherwise. }
\end{cases}
\end{align}
Note that $\tau_{a,c}$ is measurable. In particular, for $\ttau_{\aa,\cc}=(\tau_{a_i,c_i})_{i\in\{1,\dots,4\}}$,
$$\E\exp(\la F_\bb(L_\lambda[\ttau_{\aa,\cc}]))=\P(L_\lambda[\ttau_{\aa,\cc}](\LL)>\bb),$$
where we put $\aa<\bb$ for vectors $\aa=(a_1,\ldots,a_4)$, $\bb=(b_1,\ldots,b_4)\in\R^4$ if $a_i< b_i$ for all $i\in\{1,\ldots,4\}$. 
This describes the probability that 
more than $\lambda b_i$ users experience a quality of connection 
of at most $c_i$ for a period of time of more than $a_i$ for all $i\in\{1,\dots,4\}$.

\medskip
In the following, it will be convenient to consider functions $F:\, \mc{M}(\BBB)^4\to[-\infty,\infty)$ that are compatible with suitable discretizations of $\BBB$. To be more precise, we work with triadic discretizations of $W$ and ${I}$ and therefore introduce the sets $\mathbb{B}=\{3^{-m}:\,m\ge1\}$. A triadic discretization is chosen to ensure that spatially the origin is at the center of a sub-cube and that $W$ is a union of sub-cubes of the form $\L_\delta(\zeta)=\zeta+[-\delta r,\delta r]^d$ with $\zeta\in\delta 2 r\Z^d$ and $\de\in\mathbb{B}$. The space and time discretizations are given by 
$$W_\delta=\delta 2 r\Z^d\cap W\hspace{1cm}\text{and}\hspace{1cm}I_\de=\de T(\Z+\tfrac{1}{2})\cap I.$$
Now consider two operations relating the discretized path space $\Pi_\de=W_\de^{I_\de}$ of functions mapping from $I_\de$ to $W_\de$ to the continuous space $\BBB$.
Note that $\mc{M}(\Pi_\de)$ can be identified with $[0,\infty)^{\Pi_\de}$.
We discretize $x\in\BBB$ by evaluating $x$ at discrete times in $I_\de$ and spatially moving $x$ to the centers of sub-cubes, i.e.~let us denote the discretized path $\rr(x)\in\Pi_\de$ by 
$$\rr: \BBB\to\Pi_\de\hspace{1cm}x\mapsto(\rr(x_t))_{t\in I_\de}$$
where $\rr(x_{t})$ denotes the shift of $x_t\in W$ to its nearest sub-cube center in $W_\de$. 
For $\nuu\in\mc{M}(\BBB)$, the mappings $\rr$ also induces an image measure, which we will denote by 
\begin{align}
\label{discMeasure}
\nuu^\rr=\nuu\circ\rr^{-1}\in\mc{M}(\Pi_\de).
\end{align}

Second, we can embed a discretized path $z\in \Pi_\de$ as a step function into $\BBB$. That is, 
\begin{align}
\label{Iota}
\imath: \Pi_\de\to\BBB\hspace{1cm}u\mapsto\Big(\sum\nolimits_{i=0}^{\de^{-1}T-1}\one_{[i\de T,(i+1)\de T)}(t)u_{\tfrac{2i+1}{2}\de T}\Big)_{t\in I_\de}.
\end{align}
Again, for $\nuu\in\mc{M}(\Pi_\de)$, the mapping $\imath$ induces an image measure, which we will denote by 
$$\nuu^\imath=\nuu\circ\imath^{-1}\in\mc{M}(\BBB).$$

\medskip
For $\delta\in\mathbb{B}$ we say that a function $F:\, \mc{M}(\BBB)^4\to[-\infty,\infty)$ is \emph{$\delta$-discretized} if $F((\nuu^\rr)^\imath)=F(\nuu)$ holds for all $\nuu\in\mc{M}(\BBB)^4$. For instance, the functions $F(\nuu)$ as defined in \eqref{ExF} are $\delta$-discretized for every $\delta\in\mathbb{B}$.

\medskip
Our main result is a large deviation analysis of the quantities $F(L_\lambda[\ttau])$. Large-deviation results in the context of wireless networks have already been considered in literature~\cite{ldpInt,isSir}, but to the best of our knowledge our work is the first incorporating both mobility as well as relaying.
Since we obtain a level-2 large deviation result, the relative entropy plays an important r\^ole. For $\nuu,\nuu'\in \mc{M}(\LL)$ the relative entropy is defined by
$$h(\nuu|\nuu')=\int f(x)\log f(x)\nuu'(\d x)-\nuu(\LL)+\nuu'(\LL)$$
if the density $\d\nuu/d\nuu'=f$ exists and $h(\nuu|\nuu')=\infty$ otherwise. 
Let us write $\usc$for upper semicontinuous and $\lsc$for lower semicontinuous. 

\begin{Theorem}
\label{LDP}
Let $\tau_i:\, \BB\to[0,\infty)$, for $i\in\{1,\dots,4\}$, be bounded, measurable and decreasing functions that map trajectories $\g$ to zero if $\g_t\ge c_+$ for all $t\in I$. Further, let $F:\, \mc{M}(\BBB)^4\to[-\infty,\infty)$ be an increasing function that is $\delta$-discretized for some $\delta\in\mathbb{B}$, bounded from above, and maps the vector of zero measures to $-\infty$. If the $\tau_i\circ\imath$ are u.s.c.~as functions on $[0,\infty)^{I_\de}$ and $\nuu\mapsto F(\nuu^\imath)$ is u.s.c.~as a function on $\mc{M}(\Pi_\de)^4$, then
$$\limsup_{\lambda\to\infty}\frac1\lambda\log\E\exp(\la F(L_{\lambda}[\ttau]))\le-\inf_{\nuu\in\mc{M}(\LL)}\big\{h(\nuu|\muu)-F(\nuu[\ttau])\big\},$$
whereas if the $\tau_i\circ\imath$ are l.s.c.~as functions on $[0,\infty)^{I_\de}$ and $\nuu\mapsto F(\nuu^\imath)$ is l.s.c.~as a function on $\mc{M}(\Pi_\de)^4$, then
$$\liminf_{\lambda\to\infty}\frac1\lambda\log\E\exp(\la F(L_{\lambda}[\ttau]))\ge-\inf_{\nuu\in\mc{M}(\LL)}\big\{h(\nuu|\muu)-F(\nuu[\ttau])\big\}.$$
\end{Theorem}
Let us note that the semicontinuity properties of $\nuu\mapsto F(\nuu^\imath)$ and $\tau_i\circ\imath$ can be checked on finite-dimensional spaces due to our discretization assumption. This is much simpler than considering $F$ and $\tau_i$ on their infinite-dimensional domains. 

\medskip
As a special case of Theorem \ref{LDP} we obtain the rate of decay for the frustration probabilities $\P(L_{\lambda}[\ttau_{\aa,\cc}](\LL)>\bb)$ where $\ttau_{\aa,\cc}$ is defined as in \eqref{ExF2}. Furthermore, we put $[0,{\bf T})=[0,T)^4$ and $[0,\cc_+)=[0,c_+)^4$. 
\begin{Corollary}\label{Cor1}
Let $\aa\in[0,{\bf T})$, $\bb\in\R^4$ and $\cc\in[0,\cc_+)$. Then,
$$\lim_{\lambda\to\infty}\frac1\lambda\log\P(L_{\lambda}[\ttau_{\aa,\cc}](\LL)>\bb)=-\inf_{\nuu:\,\nuu[\ttau_{\aa,\cc}](\LL)> \bb}h(\nuu|\muu).$$
\end{Corollary}

Finally, we provide a formalization of the observation that the probability of unlikely frustration events decays at an exponential speed. 

\begin{Corollary}
\label{Cor2}
Let $\aa\in[0,{\bf T})$, $\bb\in\R^4$, $\cc\in[0,\cc_+)$ and assume that $((1+\e)\muu)[\ttau_{\aa,\cc}](\LL)\le\bb$ for some $\e>0$. Then,
$$\limsup_{\lambda\to\infty}\frac1\lambda\log\P(L_\lambda[\ttau_{\aa,\cc}](\LL)>\bb)<0.$$
\end{Corollary}
The remainder of the paper is organized as follows. First, Section~\ref{outSec} provides an outline of the proof of Theorem~\ref{LDP} by stating two important auxiliary results, Propositions~\ref{translUpThm} and~\ref{LDP_Discr}. They are proved in Section~\ref{prop22Sec}. 
One important ingredient in the proof of Theorem \ref{LDP} is a sprinkling argument that is established in Section~\ref{prelSec}. 
Section~\ref{LDPSec} concludes the proof of Theorem~\ref{LDP}, whereas Corollaries~\ref{Cor1} and~\ref{Cor2} are established in Section~\ref{cor1Sec}. Finally, Section~\ref{simSec} provides selected simulation results.

\section{Outline of the proof of Theorem~\ref{LDP}}
\label{outSec}
The mathematical analysis of relay-based communications is substantially less technical if we discretize the possible user locations. To be more precise, users are no longer distributed according to $\muu$ but according to $\muu^\rr$ as defined in~\eqref{discMeasure}.
In other words, we subdivide the path space into cylinder sets and assume that at times $I_\de$ all users are located at the sites in $W_\de$. By the assumptions on $\muu$ we have that
$$\k_\de=\max_{u\in \Pi_\de}\muu^\rr(u)$$
tends to zero as $\de$ tends to zero.

Let us introduce the analogue of $\nuu[\ttau]$, as given in \eqref{nuTauDef}, in the discretized setting. For a general $\nuu\in\mc{M}(\Pi_\de)$ and a general bounded $\tau:[0,\infty)^{I_\de}\to[0,\infty)$, 
$\nuu^{\ms{up}}[\tau]$ is given as a measure in $\mc{M}(\Pi_\de)$ via 
$$\frac{\d\nuu^{\ms{up}}[\tau]}{\d\nuu}(u)=\tau((R(u_t,o,\nuu_t))_{t\in I_\de})$$
and similarly for $\nuu^{\ms{up-dir}}[\tau]$, $\nuu^{\ms{do}}[\tau]$ and $\nuu^{\ms{do-dir}}[\tau]$. Finally we put, 
\begin{align*}
\nuu[\ttau]=\big(\nuu^{\ms{up}}[\tau_1],\nuu^{\ms{up-dir}}[\tau_2],\nuu^{\ms{do}}[\tau_3],\nuu^{\ms{do-dir}}[\tau_4]\big)
\end{align*}
where $\ttau=(\tau_i)_{i\in\{1,\dots,4\}}$.

The following proposition establishes dominance relationships between $\tau$-frustrated users with respect to $\nuu$ and $\nuu^\rr$ for small values of the discretization parameter $\delta$. 
\begin{Proposition}
\label{translUpThm}
Let $\e>0$, then there exists $\delta'=\delta'(\varepsilon)\in\mathbb{B}$ such that for all $\delta\in\mathbb{B}\cap(0,\delta')$, $\nuu\in\mathcal{M}(\LL)$, and $\tau_i:\, \BB\to[0,\infty)$ bounded and decreasing for all $i\in\{1,\dots,4\}$, 
$$[(1-\e)\nuu^\rr][\ttau\circ\imath]\le(\nuu[\ttau])^\rr\le [(1+\e)\nuu^\rr][\ttau\circ\imath].$$
\end{Proposition}

Working in the discrete setting simplifies the situation substantially. Instead of Poisson point processes on $\LL$, we can consider independent Poisson random variables attached to every element of the path grid $\Pi_\delta$. In particular, 
the relative entropy for the discretized setting is given by
$$h(\nuu|\muu^\rr)=\sum_{u\in \Pi_\delta}h\big(\nuu(u)|\muu^\rr(u)\big).$$
where for $a\ge0$ and $b>0$ we write $h(a|b)=a\log\tfrac ab-a+b$.
\begin{Proposition}
\label{LDP_Discr}
Let $0<\a<2$ and $\tau_i:\, [0,\infty)^{I_\de}\to[0,\infty)$, for $i\in\{1,\dots,4\}$, be bounded, measurable and decreasing functions which map trajectories $\g$ to zero if $\g_t\ge c_+$ for all $t\in I_\de$. Further, let $F:\mc{M}(\Pi_\de)^4\to[-\infty,\infty)$ be any increasing measurable function that is bounded from above and maps the vector of zero measures to $-\infty$. If $F$ and $\tau_i$ are u.s.c., then
$$\limsup_{\lambda\to\infty}\frac1\lambda\log\E\exp(\la F((\a L^\rr_{\lambda})[\ttau]))\le-\inf_{\nuu\in\mathcal{M}(\Pi_\delta)}\big\{h(\nuu|\muu^\rr)-F((\a \nuu)[\ttau])\big\}$$
whereas if $F$ and $\tau_i$ are l.s.c., then
$$\liminf_{\lambda\to\infty}\frac1\lambda\log\E\exp(\la F((\a L^\rr_{\lambda})[\ttau]))\ge-\inf_{\nuu\in\mathcal{M}(\Pi_\delta)}\big\{h(\nuu|\muu^\rr)-F((\a \nuu)[\ttau])\big\}.$$
\end{Proposition}
The difficulty of the proof of Proposition~\ref{LDP_Discr} lies in the discontinuity of the function $\nu\mapsto\nu[\ttau]$. Indeed, if the number of users on a certain site tends to zero, then in the limit other users cannot relay via this site. This might lead to a sudden drop in the QoS and therefore to a sudden increase of frustrated users. Hence, we have to deal with the continuity problems arising from configurations that exhibit sites with a small but positive number of users. A standard approach to deal with such pathological events would be to use the method of exponential approximations~\cite[Section 4.2.2]{dz98}. However, on an exponential scale, having a small but positive number of users on a certain site is not substantially less probable then having no users on this site. Therefore, exponential approximation does not seem to be an appropriate tool. We will use instead the sprinkling technique from~\cite{sprinkling}. That is, by increasing the Poisson intensity slightly, we add a small number of additional users in a way that after the sprinkling every occupied site contains a number of users that is of the same order as the Poisson intensity. We show that the assumption of observing a sprinkling of the desired kind comes at negligible cost on the exponential scale and that on the resulting configurations the map $\nu\mapsto\nu[\ttau]$ exhibits the desired continuity properties.

\section{Preliminaries}
\label{prelSec}
Before we come to the proof of Proposition~\ref{translUpThm} and Proposition~\ref{LDP_Discr}, we establish some preliminary results. First note, by a quick calculation, $D(\xi,\eta,L_{\lambda,t})=c_+$ if $L_{\la,t}(W)\le \b_o=\min\{1,\r_+^{-1}\ell_{\ms{min}}\ell_{\ms{max}}^{-1}\}$, where we put $\ell_{\ms{min}}=\min_{\xi,\eta\in W} \ell(|\xi-\eta|)$ and $\ell_{\ms{max}}=\max_{\xi,\eta\in W} \ell(|\xi-\eta|)$. 

\subsection{Monotonicity and continuity properties of QoS trajectories}
In the first lemma, we show certain monotonicity properties of $D(\xi,\eta,\nu)$ and $R(\xi,\eta,\nu)$ w.r.t.~the measure $\nu\in\mc{M}(W_\delta)$ for the sites of $W_\delta$ that have measure zero under $\nu$.
In the following, we write  $X^{\lambda}_\delta=\lambda L_\lambda^\rr$ and $V(\nu)=\{\zeta\in W_\delta:\, \nu(\zeta)=0\}$.

\begin{Lemma}
\label{gngnpCor}
Let $\delta\in\mathbb{B}$ and $\xi,\eta\in W_\delta$ be arbitrary. 
\begin{enumerate}
\item[(i)] If $\nu,\nu'\in\mc{M}(W_\delta)$ are such that $\nu\le\nu'$, then $D(\xi,\eta,\nu')\le D(\xi,\eta,\nu)$.
\item[(ii)] If $\nu,\nu'\in\mc{M}(W_\delta)$ are such that $\nu<\nu'$  and $D(\xi,\eta,\nu')<c_+$, then $D(\xi,\eta,\nu')< D(\xi,\eta,\nu)$.
\item[(iii)] If $\nu,\nu'\in\mc{M}(W_\delta)$ are such that $\nu\le\nu'$ and $V(\nu)=V(\nu')$, then $R(\xi,\eta,\nu')\le R(\xi,\eta,\nu)$.
\item[(iv)] If $\nu,\nu'\in\mc{M}(W_\delta)$ are such that $\nu<\nu'$, $V(\nu)=V(\nu')$ and $R(\xi,\eta,\nu')<c_+$, then $R(\xi,\eta,\nu')< R(\xi,\eta,\nu)$.
\item[(v)] If $\lambda'\ge\lambda>0$ and $\sigma\in(0,1)$ are such that $\tfrac{X^{\lambda'}_\delta(\Pi_\delta)}{X^{\lambda}_\delta(\Pi_\delta)}\le1+\tfrac{\ell_{\ms{min}}(1-\sigma)}{\ell_{\ms{max}}\sigma}$, then, almost surely, $D(\xi,\eta,L_{\lambda,t}^\rr)\le D(\xi,\eta,\sigma L_{\lambda',t}^\rr)$ for every $t\in I_\delta$.
\end{enumerate}
\end{Lemma}
\begin{proof}
First, we note that 
$$\frac{\SIR(\xi,\eta,\nu')}{\SIR(\xi,\eta,\nu)}=\frac{\nu(\ell(|\cdot-\eta|))}{\nu'(\ell(|\cdot-\eta|))},$$
so that the monotonicity properties of $g$ imply the first two claims. Clearly, this monotonicity also extends to expressions of the form $\Gamma(\xi,\eta,\zeta,\nu)$ with $\xi,\eta,\zeta\in W_\delta$. Moreover, under the additional condition $V(\nu)=V(\nu')$ the measures $\nu$ and $\nu'$ have the same zero-sets, which gives claims (iii) and (iv). 

Finally, the last of the asserted inequalities is equivalent to 
$$\frac{L_{\lambda',t}^\rr(\ell(|\cdot-\eta|))-L_{\lambda,t}^\rr(\ell(|\cdot-\eta|))}{L_{\lambda,t}^\rr(\ell(|\cdot-\eta|))}\le \frac{1-\sigma}{\sigma}.$$
This time, we obtain that 
\begin{align*}
\frac{L_{\lambda',t}^\rr(\ell(|\cdot-\eta|))-L_{\lambda,t}^\rr(\ell(|\cdot-\eta|))}{L_{\lambda,t}^\rr(\ell(|\cdot-\eta|))}\hspace{-0.05cm}=\hspace{-0.05cm}\frac{\tfrac{\lambda}{\lambda'}X^{\lambda'}_{\delta,t}(\ell(|\cdot-\eta|))-X^{\lambda}_{\delta,t}(\ell(|\cdot-\eta|))}{X^{\lambda}_{\delta,t}(\ell(|\cdot-\eta|))}\hspace{-0.05cm}\le\hspace{-0.05cm}\frac{\ell_{\ms{max}}}{\ell_{\ms{min}}} \frac{X^{\lambda'}_{\delta,t}(W_\delta)-X^{\lambda}_{\delta,t}(W_\delta)}{X^{\lambda}_{\delta,t}(W_\delta)},
\end{align*}
as required.
\end{proof}

In the next lemma, we relate essential suprema w.r.t.~path measures and their time projections. We will write $A^c$ to indicate the complement of a set $A$.
\begin{Lemma}
\label{EsssupLemma} 
Let $\nuu\in\mc{M}(\LL)$, $x\in\LL$ and $t\in I$, then
$$\nuu_t\text{-}\esssup_{\eta\in W}\G(x_t,\eta,o,\nuu_t)=\nuu\text{-}\esssup_{y\in\LL}\G(x_t,y_t,o,\nuu_t).$$
\end{Lemma}
\begin{proof}
We first show $\ge$. Let $N_t$ be such that $\nuu_t(N_t)=0$ and define $N=\{\g\in\LL: \g_t\in N_t\}$. In particular $\nuu(N)=0$ and it suffices to show that 
$$\sup_{\eta\in (N_t)^c}\G(x_t,\eta,o,\nuu_t)\ge\sup_{y\in N^c}\G(x_t,y_t,o,\nuu_t).$$
But this is trivially true. For the converse, $\le$, let $N$ be such that $\nuu(N)=0$. We define $N_t=(\pi_t(N^c))^c$ and note that $\pi_t^{-1}(N_t)\subset N$. Indeed, suppose that $\g\in N^c$, then $\g_t\in \pi_t(N^c)$ and thus $\g_t\notin N_t$ which implies that $\g\notin \pi_t^{-1}(N_t)$. Since $\pi_t$ is continuous, 
by \cite[Theorem 13.2.6]{Dudley}, $N_t$ is a universally measurable set and there exist Borel measurable sets $A,B\subset W$ such that $A\subset N_t\subset B$ and $\nuu(\pi_t^{-1}(A))=\nuu(\pi_t^{-1}(B))$. This implies, that $\nuu_t(B)=\nuu_t(A)\le \nuu(N)=0$. From this it follows that $N_t$ is a $\nuu_t$ nullset. Hence, it suffices to show that 
$$\sup_{\eta\in B^c}\G(x_t,\eta,o,\nuu_t)\le\sup_{y\in N^c}\G(x_t,y_t,o,\nuu_t).$$
But this is also true since by construction, for every $\eta\in B^c\subset(N_t)^c=\pi_t(N^c)$ there exists a $y\in N^c$ such that $\eta=y_t$.
\end{proof}
\begin{Remark}
Lemma~\ref{EsssupLemma} remains true if the uplink $\G$ is replaced by the downlink $\G$.
\end{Remark}
Next, we transfer regularities of paths supported by $\nuu\in\mc{M}(\LL)$ to the QoS trajectories $\bar D$ and $\bar R$. 
\begin{Lemma}
\label{PathRegularities}
Let $\nuu\in\mc{M}(\LL)$, then for any $x\in\LL$ we have that
\begin{enumerate}
\item[(i)] $t\mapsto D(x_t,y_t,\nuu_t)$ is Lipschitz continuous for all $y\in\LL$ and 
\item[(ii)] $t\mapsto R(x_t,o,\nuu_t)$, $t\mapsto R(o,x_t,\nuu_t)$ are Lipschitz continuous.
\end{enumerate}
\end{Lemma}
\begin{proof}
First note that if $\nuu(\LL)=0$, then by the definition of $g$, $\bar D$ and $\bar R$ are constant and hence Lipschitz continuous. Let $\nuu(\LL)>0$. Next we show that $t\mapsto \SIR(x_t,y_t,\nuu_t)$ is Lipschitz continuous. Indeed, since $x,y$ and $\ell$ are assumed to be Lipschitz continuous, comparing the numerator in $\SIR$ gives
$$|\ell(|x_s-y_s|)-\ell(|x_t-y_t|)|\le J_2(|x_s-x_t|+|y_s-y_t|)\le 2J_2J_1|s-t|$$
which tends to zero as $s$ tends to $t$. For the denominator, using the above, we have 
$$|\nuu_s(\ell(|\cdot-y_s|))-\nuu_t(\ell(|\cdot-y_t|))|=|\nuu(\ell(|\pi_s(\cdot)-y_s|))-\nuu(\ell(|\pi_t(\cdot)-y_t|))|\le 2J_2J_1\nuu(\LL)|s-t|.$$
Using this we can conclude
$$|\SIR(x_s,y_s,\nuu_s)-\SIR(x_t,y_t,\nuu_t)|\le(\frac{\ell_{\ms{max}}}{\ell_{\ms{min}}}+1)\frac{2J_2J_1}{\nuu(\LL)\ell_{\ms{min}}}|s-t|$$
where the Lipschitz constant depends on $\nuu$ but not on $x$ and $y$.
Now, since $g$ is assumed to be Lipschitz continuous, part $(i)$ follows from the definition of $\bar D$.
For $\bar R(x,o,\nuu)$ in part $(ii)$ it suffices to show that $t\mapsto \nuu_t\text{-}\esssup_{\eta\in W}\Gamma(x_t,\eta,o,\nuu_t)$ is Lipschitz continuous since taking maxima is Lipschitz continuous. Let $\e>0$ be arbitrary. We can use Lemma~\ref{EsssupLemma} to lift the essential suprema to the path level and estimate
\begin{align}
\label{Est1}
\nuu\text{-}\esssup_{y\in\LL}\Gamma(x_t,y_t,o,\nuu_t)-\nuu\text{-}\esssup_{y\in\LL}\Gamma(x_s,y_s,o,\nuu_s)\le\Gamma(x_t,y_t,o,\nuu_t)-\Gamma(x_s,y_s,o,\nuu_s)+2\e
\end{align}
for some $y=y(x_t)\in N^c$ where $N=N(x_s)$ is a $\nuu$-nullset. Since $\G$ is given as a maximum of Lipschitz continuous functions, with parameter independent of $x$ and $y$, in the r.h.s.~of \eqref{Est1} we can further estimate
\begin{align*}
\Gamma(x_t,y_t,o,\nuu_t)-\Gamma(x_s,y_s,o,\nuu_s)\le \a|s-t|
\end{align*}
for some constant $\a>0$. Sending $\e$ to zero and using the symmetry in $s$ and $t$, this gives the Lipschitz continuity. For $\bar R(o,x,\nuu)$ the proof is analogous. 
\end{proof}

Note that by the above lemma, for $x,y\in\LL$ and $\nuu\in\mc{M}(\LL)$, $\bar D(x,y,\nuu)$, $\bar R(x,o,\nuu)$ and $\bar R(o,x,\nuu)$ are elements of $\BB$.
The following lemma establishes continuity and in particular Borel measurability for the QoS quantities as functions of Lipschitz paths. Let us write $\Vert\cdot\Vert$ for the supremum norm. 

\begin{Lemma}
\label{PathRegularities2}
Let $y\in\LL$ and $\nuu\in\mc{M}(\LL)$. Then, as mappings from $\LL$ to $\BB$
\begin{enumerate}
\item[(i)] $x\mapsto \bar D(x,y,\nuu)$ is Lipschitz continuous and 
\item[(ii)] $x\mapsto \bar R(x,o,\nuu)$ and $x\mapsto \bar R(o,x,\nuu)$ are Lipschitz continuous.
\end{enumerate}
\end{Lemma}
\begin{proof}
As above, note that if $\nuu(\LL)=0$, then by the definition of $g$, $\bar D$ and $\bar R$ are constant and hence continuous. Let $\nuu(\LL)>0$. Next we show that $x\mapsto \overline\SIR(x,y,\nuu)$ is continuous. Indeed, for any $x,x'\in\LL$, we have 
\begin{align*}
\Vert\SIR(x,y,\nuu)-\SIR(x',y,\nuu)\Vert&\le \frac{1}{\ell_{\ms{min}}\nuu(\LL)}\sup_{t\in I}|\ell(|x_t-y_t|)-\ell(|x'_t-y_t|)|\le\frac{J_2}{\ell_{\ms{min}}\nuu(\LL)} \Vert x-x'\Vert
\end{align*}
where the Lipschitz parameter is independent of $y$.
Now, since $g$ is assumed to be Lipschitz continuous, part $(i)$ follows from the definition of $\bar D$.
For $\bar R(x,o,\nuu)$ in part $(ii)$ it suffices to show that $x\mapsto (\nuu_t\text{-}\esssup_{\eta\in W}\Gamma(x_t,\eta,o,\nuu_t))_{t\in I}$ is Lipschitz continuous. 
Let $\e>0$ be arbitrary, $t\in I$ and $x,x'\in \LL$ then we have
\begin{align}
\label{Est3}
\nuu_t\text{-}\esssup_{\eta\in W}\Gamma(x_t,\eta,o,\nuu_t)-\nuu_t\text{-}\esssup_{\eta\in W}\Gamma(x'_t,\eta,o,\nuu_t)\le\Gamma(x_t,\eta,o,\nuu_t)-\Gamma(x'_t,\eta,o,\nuu_t)+2\e
\end{align}
for some $\eta=\eta(x_t)\in N^c$ where $N=N(x'_t)$ is a $\nuu_t$-nullset. Since $\G$ is given as a maximum of Lipschitz continuous functions, with parameter independent of $\eta$, in the r.h.s.~of \eqref{Est3} we can further estimate
\begin{align*}
\label{Est4}
\Gamma(x_t,\eta,o,\nuu_t)-\Gamma(x'_t,\eta,o,\nuu_t)\le \a|x_t-x'_t|
\end{align*}
for some constant $\a>0$. Sending $\e$ to zero, using the symmetry in $x$ and $x'$ and taking suprema over $t$, gives the Lipschitz continuity. For $\bar R(o,x,\nuu)$ the proof is analogous. 
\end{proof}

The following results are for the discretized setting. Note, that in case of relayed communication, the QoS of a given user is very sensitive to the distribution of the surrounding users. This is due to the fact, that the disappearance of possible relays might lead to a sudden decrease in QoS. This is captured by the fact, that the function $\nuu\mapsto \bar R(x,o,\nuu)$ is only l.s.c.
\begin{Lemma}
\label{upSemContLem} 
For all $u,v\in\Pi_\de$, the maps $\nuu\mapsto \bar D(u,v,\nuu)$, $\nuu\mapsto \bar R(u,o,\nuu)$ and $\nuu\mapsto \bar R(o,u,\nuu)$ from $\mc{M}(\Pi_\de)\to [0,\infty)^{I_\de}$ are continuous, l.s.c.~and l.s.c.~respectively.
\end{Lemma}
\begin{proof}
It suffices to show that the maps $\nu\mapsto D(\xi,\eta,\nu)$, $\nu\mapsto R(\xi,o,\nu)$ and $\nu\mapsto R(o,\xi,\nu)$, as maps from $\mc{M}(W_\de)$ to $[0,\infty)$, are continuous, respectively l.s.c. and l.s.c., for all $\xi\in W$. Let $\nu_n$ be a sequence in $\mc{M}(W_\de)$ which tends to $\nu_*$. First note that if $\nu_*(W_\de)=0$, then there exists $m\in\N$ such that $\nu_n(W_\de)\le \b_o$ for all $n\ge m$, which implies that $R(o,\xi,\nu_n)=R(o,\xi,\nu_*)=R(\xi,o,\nu_n)=R(\xi,o,\nu_*)=D(\xi,o,\nu_n)=D(\xi,o,\nu_*)=c_+$ for all $n\ge m$. Second, assume that $\nu_*(\eta)>0$ for some $\eta\in W_\de$, then the continuity of $\nu\mapsto\SIR(\xi,\eta,\nu)$ at $\nu_*$ implies the continuity of $D$ and $\G$ at $\nu_*$. This is the first part of the statement. Moreover, since we work in the discrete setting, the essential supremum in the definition of $R(\xi,o,\nu)$ and $R(o,\xi,\nu)$ can always be written as a maximum and, for the uplink, it suffices to prove that $\nu\mapsto\max_{\eta\in W}\one\{\nu(\eta)>0\}\Gamma(\xi,\eta,o,\nu)$ is l.s.c. Furthermore, since $\mc{M}(W_\de)$ is finite dimensional, there exists $m\in\N$ such that $\nu_n(\eta)>0$ for all $n\ge m$ and all $\eta$ with $\nu_*(\eta)>0$. For such $n$ we have 
$$\max_{\eta\in W}\one\{\nu_n(\eta)>0\}\Gamma(\xi,\eta,o,\nu_n)\ge\max_{\eta\in W}\one\{\nu_*(\eta)>0\}\Gamma(\xi,\eta,o,\nu_n)$$
where the r.h.s.~tends to $\max_{\eta\in W}\one\{\nu_*(\eta)>0\}\Gamma(\xi,\eta,o,\nu_*)$ by continuity. But this is lower semicontinuity. For the relayed downlink, analogue arguments apply.
\end{proof}
Let us call a function $f:\,[0,\infty)^m\to[-\infty,\infty)$ decreasing, if $f$ is decreasing w.r.t.~the partial order on $[0,\infty)^m$ given by $x\le y$ if and only if $x_i\le y_i$ for all $1\le i\le m$. $f$ is called increasing if $-f$ is decreasing. Further we call a function $g:\,[0,\infty)^n\to\R^m$ u.s.c~if $g$ is u.s.c.~as a mapping in every coordinate $1\le i\le m$ in the image space. $g$ is called l.s.c.~if $-g$ is u.s.c. We will need the following general auxiliary result on the composition of semicontinuous functions.
\begin{Lemma}
\label{uscLem}
Let $f:\,[0,\infty)^m\to[-\infty,\infty)$ and $g:\,[0,\infty)^n\to[0,\infty)^m$, where $g$ maps bounded sets to bounded sets and $n,m\in\N$. 
\begin{enumerate}
\item[(i)] If $g$ is l.s.c.~and $f$ is decreasing and u.s.c., then $f\circ g$ is u.s.c.
\item[(ii)] If $g$ is u.s.c.~and $f$ is decreasing and l.s.c., then $f\circ g$ is l.s.c.
\item[(iii)] If $g$ is u.s.c.~and $f$ is increasing and u.s.c., then $f\circ g$ is u.s.c.
\item[(iv)] If $g$ is l.s.c.~and $f$ is increasing and l.s.c., then $f\circ g$ is l.s.c.
\end{enumerate}
\end{Lemma}
\begin{proof}
We only prove the first claim, since the others can be shown using similar arguments. Assume that $\nu_k\in[0,\infty)^n$ are such that $\lim_{k\to\infty}\nu_k=\nu_*$ for some $\nu_*\in [0,\infty)^n$. Then, we have to show that $\limsup_{k\to\infty}f(g(\nu_k))\le f(g(\nu_*))$. After passing to a subsequence, we may replace the $\limsup$ on the l.h.s.~by a $\lim$. Since $g$ maps bounded sets to bounded sets, we may pass to a further subsequence and assume that $g(\nu_k)$ converges to some $\nu'\in[0,\infty)^m$. Since $g$ is l.s.c.~$\nu'_i=\lim_{k\to\infty}g_i(\nu_k)\ge g_i(\nu_*)$ for all $1\le i\le m$ and thus, since $f$ is decreasing, we have $f(\nu')\le f(g(\nu_*))$. Hence, using the upper semicontinuity of $f$, we arrive at
$$\lim_{k\to\infty}f(g(\nu_k))\le f(\nu')\le f(g(\nu_*)).$$
as required.
\end{proof}
\begin{Remark}
\label{USC}
By part (i) of the above lemma we have, under the assumptions that $\tau$ is u.s.c.~and decreasing, that
the map $\nuu\mapsto\tau(\bar R(u,o,\nuu))$ is u.s.c., where we also use Lemma \ref{upSemContLem}. Moreover, part (iii) of the above lemma can be used also to show that the map $\nuu\mapsto F(\nuu[\tau])$ appearing in 
Proposition \ref{LDP_Discr} is u.s.c.~for any increasing and u.s.c.~function $F$.
\end{Remark}

\subsection{Sprinkling construction}\label{Sprinkling}
As mentioned in the paragraph after Proposition~\ref{LDP_Discr}, the main difficulty in analyzing the empirical measures $L_{\lambda}^\rr[\ttau]$ comes from the discontinuity of the indicators $\one\{L^\rr_{\lambda}(u)>0\}$, $u\in \Pi_\delta$. In other words, the configurations that constitute obstructions in applying the contraction principle from large-deviation theory~\cite[Theorem 4.2.1]{dz98} are those exhibiting $\delta$-discretized trajectories with a small but non-zero number of users. Now, we show that a small increase in the intensity of the Poisson point process provides us with a sufficient amount of additional randomness allowing us to exclude such pathological configurations. In other words, we make use of a sprinkling argument in the spirit of~\cite{sprinkling}.

In order to perform the sprinkling operation with parameter $\e_0\in(0,1)$, we define $X^{\lambda}_\delta$ and $X^{\lambda'}_\delta$ as above with $\lambda'=(1+\e_1)\lambda$, where we put
$\e_1=2\varepsilon_0\k_{\delta}^{-1}$. In the following, we always assume that $\delta\in\mathbb{B}$ is sufficiently small to ensure that $\kappa_\delta\le1$.
Now, for $u\in \Pi_\delta$ we put
$$Q=\{u\in \Pi_\delta:\, L_{\lambda'}^\rr(u)\le\e_0\}\qquad\text{ and }\qquad V=\{u\in \Pi_\delta:\, L_{\lambda}^\rr(u)=0\}$$
denote the sets of all \emph{quasi-empty} and \emph{virtual} sites of $\Pi_\delta$, respectively.
We write 
$$E_{\e_0}=\{Q\subset V\}$$ 
for the event that all quasi-empty sites are virtual.
Similarly, we introduce the event 
$$E'_{\e_0}=\{X^{\lambda}_\delta(\Pi_\delta)\ge(1-\e_2)X^{\lambda'}_\delta(\Pi_\delta)\},$$ 
where $\e_2=4\e_1\b_o^{-1}\#\Pi_\delta(1+\mu^\rr(\Pi_\delta))$ and $\# \Pi_\de$ denotes the number of space-time sub-cubes in the discretization $\Pi_\de$.
Now, the following two auxiliary results, formalize the sprinkling heuristic described above.

\begin{Lemma}
\label{condProbLem}
For all sufficiently small $\e_0\in(0,1)$ there exists $\lambda_0=\lambda_0(\e_0)$ such that 
$$\P(E_{\e_0}\cap E'_{\e_0}|X^{\lambda'}_\delta)\ge\exp(-\sqrt{\e_0}\lambda)\one\{L^\rr_{\lambda'}(\Pi_\delta)\ge \b_o/2\}$$
holds almost surely for all $\lambda\ge\lambda_0$.
\end{Lemma}
\begin{proof}
The proof is based on the observation that $X^{\lambda}_\delta$ is obtained from $X^{\lambda'}_\delta$ by independent thinning with survival probability $1/(1+\varepsilon_1)$, see for example \cite[Section 5.1]{Kingman93}. In other words, for every $u\in \Pi_\delta$ there exist $N'_{u}=X^{\lambda'}_\delta(u)$ independent $\ms{Ber}(1/(1+\e_1))$-distributed random variables $\{U_{k}(u)\}_{1\le k\le N'_{u}}$ such that 
$$X^{\lambda}_\delta(u)=\sum_{k=1}^{N'_u}U_k(u).$$

In terms of the Bernoulli variables, $E_{\e_0}$ is the event that $U_k(u)=0$ holds for all $u\in Q$ and $1\le k\le N'_u$.
Now, we let $E''_{\e_0}$ denote the event that 
$X^{\lambda}_\delta(u)\ge (1-\e_1)X^{\lambda'}_\delta(u)$
holds for all $u\in\Pi_\delta\setminus Q$.
Since $\#Q\le \#\Pi_\delta$, this implies that $X^{\lambda}_\delta(\Pi_\delta)$ is bounded below by
\begin{align*}
 (1-\e_1)X^{\lambda'}_\delta(\Pi_\delta\setminus Q)=(1-\e_1)(X^{\lambda'}_\delta(\Pi_\delta)-X^{\lambda'}_\delta(Q))\ge(1-\e_1)(X^{\lambda'}_\delta(\Pi_\delta)-\#\Pi_\delta\e_0\lambda').
\end{align*}
If additionally $\{L^\rr_{\lambda'}(\Pi_\delta)\ge \b_o/2\}$ occurs, then we may extend the above estimation as follows
\begin{align*}
X^{\lambda}_\delta(\Pi_\delta)\ge(1-\e_1)(1-2\#\Pi_\delta\e_0\b_o^{-1})X^{\lambda'}_\delta(\Pi_\delta)\ge(1-\e_2)X^{\lambda'}_\delta(\Pi_\delta).
\end{align*}
Therefore, $E''_{\e_0}\cap\{L^\rr_{\lambda'}(\Pi_\delta)\ge \b_o/2\}\subset E'_{\e_0}$ and it remains to bound $\P(E_{\e_0} \cap E''_{\e_0}|X^{\lambda'}_\delta)$ from below. First note that
$$\P(E''_{\e_0}|X^{\lambda'}_\delta)\ge\P\big(U_k(u)=1\text{ for all } u\in\Pi_\de\setminus Q\text{ and }1\le k\le N'_u|X^{\lambda'}_\delta\big)\ge(1+\e_1)^{-\sum_{u\in \Pi_\de}N'_u}.$$ 
By the law of large numbers, if the $N'_u$ in $X^{\lambda'}_\delta$ is large, $\P(X^{\lambda}_\delta(u)\ge (1-\e_1)X^{\lambda'}_\delta(u)|X^{\lambda'}_\delta)$ is close to one. Together, this implies that there exists $c_1>0$ such that $\P(E''_{\e_0}|X^{\lambda'}_\delta)\ge c_1$ holds almost surely for every $\lambda\ge1$. Hence, since conditioned on $X^{\lambda'}_\delta$, $E_{\e_0}$ and $E''_{\e_0}$ are independent, we obtain that 
\begin{align*}
\P(E_{\e_0}\cap E''_{\e_0} |X^{\lambda'}_\delta)=\P(E_{\e_0}|X^{\lambda'}_\delta)\P(E''_{\e_0}|X^{\lambda'}_\delta)\ge(\e_12^{-1})^{\sum_{u\in Q}N'_u}c_1\ge(\e_12^{-1})^{\e_0\lambda'\#\Pi_\delta}c_1.
\end{align*}
In particular, observing that $-\e_0\log(\e_12^{-1})\in o(\sqrt{\e_0})$ concludes the proof.
\end{proof}
In the following, we extend the definition of $V$ to measures on $\Pi_\delta$ and define $V(\nuu)=\{u\in \Pi_\delta:\, \nuu(u)=0\}$.
Furthermore, we put 
$$E^*_{\e_0}=E_{\e_0}\cap\{V(L_\lambda^\rr)=V(L_{\lambda'}^\rr)\}.$$
\begin{Lemma}
\label{condProbLem2}
For all sufficiently small $\e_0\in(0,1)$ there exists $\lambda_0=\lambda_0(\e_0)$ such that 
$$\P(E^*_{\e_0}|X^{\lambda}_\delta)\ge \exp(-\sqrt{\e_0}\lambda)$$
holds almost surely for all $\lambda\ge \lambda_0$. 
\end{Lemma}
\begin{proof}
First, we note that $H_{\e_0}\subset E_{\e_0}$, where $H_{\e_0}$ denotes the event that $N''_u\ge\e_0\lambda'$ holds for all $u\in \Pi_\delta\setminus V(L_\lambda^\rr)$. Here $N''_u=X^{\lambda'}_\delta(u)-X^{\lambda}_\delta(u)$ is independent of $X^{\lambda}_\delta(u)$ and Poisson distributed with parameter $(\la'-\la)\muu^\rr(u)>\e_0\la'$. In particular, similarly to the proof of Lemma~\ref{condProbLem}, there exists $c_1>0$ such that $\P(H_{\e_0}|X^{\lambda}_\delta)\ge c_1$ holds almost surely for every $\la\ge1$.
Therefore, conditioned on $X^{\lambda}_\delta$, the independence of $\{V(L_\lambda^\rr)=V(L_{\lambda'}^\rr)\}$ and $H_{\e_0}$ gives that
\begin{align*}
\P(\{V(L_\lambda^\rr)=V(L_{\lambda'}^\rr)\}\cap H_{\e_0} |X^{\lambda}_\delta)=\P(V(L_\lambda^\rr)=V(L_{\lambda'}^\rr)|X^{\lambda}_\delta)\P(H_{\e_0}|X^{\lambda}_\delta).
\end{align*}
Since, the r.h.s.~is bounded below by $\exp(-\muu^\rr(\Pi_\delta)(\lambda'-\lambda))c_1$, we conclude the proof.
\end{proof}

\subsection{Relative entropies under linear perturbation}
Finally, it will be convenient to quantify the impact of multiplication of measures by scalars on relative entropies. 

\begin{Lemma}
\label{ent0Lem}
Let $a>0$ and $\nuu\in\mc{M}(\LL)$ be arbitrary. Then,
$$h(a\nuu|\muu)=ah(\nuu|\muu)+a\log(a)\nuu(\LL)+(1-a)\muu(\LL).$$
\end{Lemma}
\begin{proof}
The claim is trivial if $\nuu$ is not absolutely continuous with respect to $\muu$. Otherwise, writing $f=\d\nuu/\d\muu$ we have that 
\begin{align*}
h(a\nuu|\muu)=a\Big(\int \log f \nuu(\d x)-\nuu(\LL)+\muu(\LL)\Big)+a\log(a) \nuu(\LL)+(1-a)\muu(\LL),
\end{align*}
as required.
\end{proof}

As a corollary, we obtain the following bounds on $h(a \nuu|\muu)$.
\begin{Corollary}
\label{ent0Cor}
Let $\nuu\in\mc{M}(\LL)$ and $\e\in(0,1/2)$ be arbitrary. Then,
$$h((1+\e)\nuu|\muu)\le (1+3\e)h(\nuu|\muu)+3\e\muu(\LL).$$
and
$$h((1-\e)\nuu|\muu)\ge(1-3\e)h(\nuu|\muu)-3\e\muu(\LL).$$
\end{Corollary}
\begin{proof}
By Lemma~\ref{ent0Lem} the claims are equivalent to 
$$(1+\e)\nuu(\LL)\log(1+\e)\le2\e(h(\nuu|\muu)+2\muu(\LL)).$$
and
$$(1-\e)\nuu(\LL)\log(1-\e)\ge-2\e(h(\nuu|\muu)+2\muu(\LL)).$$
First, by Jensen's inequality and an elementary optimization exercise,
$$\nuu(\LL)\le h(\nuu(\LL)|\muu(\LL))+2\muu(\LL).$$
Combining this inequality with the bounds
$(1+\e)\log(1+\e)\le 2\e$ and $(1-\e)\log(1-\e)\ge -2\e$ completes the proof.
\end{proof}
\begin{Remark}
Lemma~\ref{ent0Lem} and Corollary~\ref{ent0Cor} remain true if $\LL$ and $\muu$ are replaced by $\Pi_\delta$ and $\muu^\rr$, respectively.
\end{Remark}

\section{Proof of Proposition~\ref{translUpThm} and Proposition~\ref{LDP_Discr}}
\label{prop22Sec}
\subsection{Proof of Proposition~\ref{translUpThm}}
First note, that for all $u\in \Pi_\delta$ with $\nuu^\rr(u)=0$ we have $((1\pm\e)\nuu^\rr)[\ttau\circ\imath](u)=0$ and $(\nuu[\ttau])^\rr(u)=0$ and hence the inequalities are trivially satisfied.
Now, fix $\e>0$ and assume $u\in \Pi_\delta$ with $\nuu^\rr(u)>0$. 

\medskip
Let us introduce $C(\eta,\nu)\in\{D(\eta,o,\nu),R(\eta,o,\nu),D(o,\eta,\nu),R(o,\eta,\nu)\}$ for the different forms of communication where $\eta\in W$ and $\nu\in\mc{M}(W)$. Denote $\bar C(x,\nuu)$ the associated trajectory.

\medskip
We first prove the the upper bound. It suffices to find $\delta'=\delta'(\e)\in\mathbb{B}$ such that for all $\de\in\mathbb{B}\cap(0,\delta')$ and all $C$ we have
$$\sup_{x\in\rr^{-1}(u)}\tau_i(\bar C(x,\nuu))\le\tau_i(\imath(C(u_t,(1+\e)\nuu^\rr_t)_{t\in I_\de}))$$
for all $i\in\{1,\dots,4\}$.
Since the $\tau_i$ are decreasing, it suffices to find $\delta'=\delta'(\e)\in\mathbb{B}$ such that for all $\de\in\mathbb{B}\cap(0,\delta')$ and $x\in \rr^{-1}(u)$
\begin{equation}\label{InequalityProp21}
\begin{split}
\bar C(x,\nuu)\ge \imath(C(u_t,(1+\e)\nuu^\rr_t)_{t\in I_\de}).
\end{split}
\end{equation}
Let us first show that for all $x\in\rr^{-1}(u)$ and $y\in\rr^{-1}(v)$ with $u\in\Pi_\de$ we have
\begin{equation}\label{CentralInequalityProp21}
\begin{split}
\overline\SIR(x,y,\nuu)\ge\imath(\SIR(u_t,v_t,(1+\e)\nuu^\rr_t)_{t\in I_\de})
\end{split}
\end{equation}
for sufficiently small $\delta$. Using the definition of SIR this is equivalent to showing for all $t\in I_\de$ that
$$\frac{\ell(|u_t-v_t|)\nuu_s(\ell(|\cdot-y_s|))}{\ell(|x_s-y_s|)\nuu_t^\rr(\ell(|\cdot-v_t|))}-1\le\e$$
for all $t-\tfrac{\de T}{2}\le s<t+\tfrac{\de T}{2}$.
Note, the left hand side can be estimated as follows
\begin{equation}\label{Eps}
\begin{split}
&\Big|\frac{\ell(|u_t-v_t|)\nuu_s(\ell(|\cdot-y_s|))}{\ell(|x_s-y_s|)\nuu_t^\rr(\ell(|\cdot-v_t|))}-1\Big|\le\Big|\frac{\ell(|u_t-v_t|)\nuu_s(\ell(|\cdot-y_s|))-\ell(|x_s-y_s|)\nuu_t^\rr(\ell(|\cdot-v_t|))}{\ell(|x_s-y_s|)\nuu_t^\rr(\ell(|\cdot-v_t|))}\Big|\cr
&\hspace{1cm}\le\frac{\ell(|u_t-v_t|)}{\ell(|x_s-y_s|)}\Big|\frac{\nuu_s(\ell(|\cdot-y_s|))-\nuu_t^\rr(\ell(|\cdot-v_t|))}{\nuu_t^\rr(\ell(|\cdot-v_t|))}\Big|+\Big|\frac{\ell(|u_t-v_t|)-\ell(|x_s-y_s|)}{\ell(|x_s-y_s|)}\Big|\cr
&\hspace{1cm}\le\frac{\ell_{\ms{max}}}{\ell_{\ms{min}}^2}\frac{\sum_{w\in\Pi_\de}\int_{\rr^{-1}(w)}\nuu(\d z)|\ell(|z_s-y_s|)-\ell(|w_t-v_t|)|}{\nuu(\LL)}+\frac{J_2}{\ell_{\ms{min}}}|u_t-v_t-x_s+y_s|\cr
&\hspace{1cm}\le\frac{J_2\ell_{\ms{max}}}{\ell_{\ms{min}}^2}\sup_{w\in\Pi_\de}\sup_{z\in\rr^{-1}(w)}|z_s-y_s-w_t+v_t|+\frac{J_2}{\ell_{\ms{min}}}|u_t-v_t-x_s+y_s|\cr
&\hspace{1cm}\le \a_1\sup_{w\in\Pi_\de}\sup_{z\in\rr^{-1}(w)}|z_s-w_t|
\le \a_2\, \de\cr
\end{split}
\end{equation}
where $\a_1,\a_2$ are some constants involving also $r,T$ and $J_1$.
Since $g$ is assumed to be increasing, \eqref{CentralInequalityProp21} implies 
\begin{equation}\label{Central2InequalityProp21}
\begin{split}
\bar D(x,y,\nuu)\ge \imath(D(u_t,v_t,(1+\e)\nuu^\rr_t)_{i\in I_\de})
\end{split}
\end{equation}
for all $x\in\rr^{-1}(u)$ and $y\in\rr^{-1}(v)$.
Now, for every $C$, the inequality \eqref{InequalityProp21} can be derived from the inequality \eqref{Central2InequalityProp21}.
Indeed, for the direct up- and downlink cases \eqref{InequalityProp21} are implied by \eqref{Central2InequalityProp21} setting 
$y=v=o$ respectively $x=u=o$.
For the relayed uplink case $C(\eta,\nu)=R(\eta,o,\nu)$ we have to prove \eqref{InequalityProp21} only for the relaying component in $R(\eta,o,\nu)$ since the direct communication part we already verified. In other words, using Lemma \ref{EsssupLemma}, we show that for all $t\in I_\de$ and $x\in\rr^{-1}(u)$ we have
\begin{equation*}
\begin{split}
\nuu\text{-}\esssup_{y\in\LL}\Gamma(x_s,y_s,o,\nuu_s)\ge\nuu^\rr\text{-}\esssup_{v\in\Pi_\delta}\Gamma(u_t,v_t,o,(1+\e)\nuu_t^\rr)
\end{split}
\end{equation*}
for all $t-\tfrac{\de T}{2}\le s<t+\tfrac{\de T}{2}$.
Let us assume the supremum on the right hand side is attained in $v\in \Pi_\delta$ where necessarily $\nuu^\rr(v)>0$. Then it suffices to find $\delta'=\delta'(\e)\in\mathbb{B}$ such that for all $\delta\in\mathbb{B}\cap(0,\delta')$, $x\in\rr^{-1}(u)$, $y\in\rr^{-1}(v)$ and $t-\tfrac{\de T}{2}\le s<t+\tfrac{\de T}{2}$
\begin{equation*}
\begin{split}
\min\{D(x_s,y_s,\nuu_s),D(y_s,o,\nuu_s)\}\ge\min\{D(u_t,v_t,(1+\e)\nuu_t^\rr),D(v_t,o,(1+\e)\nuu_t^\rr)\},
\end{split}
\end{equation*}
but this can be done using \eqref{Central2InequalityProp21}.

Similarly, in the case of relayed downlink communication $C(\eta,\nu)=R(o,\eta,\nu)$, using the same argument as in the relayed uplink case, we need to show
\begin{equation*}
\begin{split}
\min\{D(o,y_s,\nuu_s),D(y_s,x_s,\nuu_s)\}\ge\min\{D(o,v_t,(1+\e)\nuu_t^\rr),D(v_t,x_t,(1+\e)\nuu_t^\rr)\}.
\end{split}
\end{equation*}
for all $x\in\rr^{-1}(u)$, $y\in\rr^{-1}(v)$, $t-\tfrac{\de T}{2}\le s<t+\tfrac{\de T}{2}$ and sufficiently small $\delta\in\mathbb{B}$. But this is also true using \eqref{Central2InequalityProp21}.

\medskip
For the lower bound, it suffices to find $\delta'=\delta'(\e)\in\mathbb{B}$ such that for all $\delta\in\mathbb{B}\cap(0,\delta')$, $x\in\rr^{-1}(u)$ and all $C$ we have
\begin{equation}\label{InequalityProp21Left}
\begin{split}
\bar C(x,\nuu)\le \imath(C(u_t,(1-\e)\nuu^\rr_t)_{t\in I_\de}).
\end{split}
\end{equation}
Again, we first show that for all $x\in\rr^{-1}(u)$ and $y\in\rr^{-1}(v)$ with $u\in\Pi_\de$ we have
\begin{equation*}\label{CentralInequalityProp21Left}
\begin{split}
\overline\SIR(x,y,\nuu)\le\imath(\SIR(u_t,v_t,(1-\e)\nuu^\rr_t)_{t\in I_\de})
\end{split}
\end{equation*}
for sufficiently small $\delta$, which is equivalent to showing for all $t\in I_\de$ that
$$1-\frac{\ell(|u_t-v_t|)\nuu_s(\ell(|\cdot-y_s|))}{\ell(|x_s-y_s|)\nuu_t^\rr(\ell(|\cdot-v_t|))}\le\e$$
for all $t-\tfrac{\de T}{2}\le s<t+\tfrac{\de T}{2}$. But this is true using again the estimate \eqref{Eps}. This implies 
\begin{equation}\label{Central2InequalityProp21Left}
\begin{split}
\bar D(x,y,\nuu)\le \imath(D(u_t,v_t,(1-\e)\nuu^\rr_t)_{i\in I_\de})
\end{split}
\end{equation}
for all $x\in\rr^{-1}(u)$ and $y\in\rr^{-1}(v)$.

For the direct up- and downlink cases \eqref{InequalityProp21Left} are implied by \eqref{Central2InequalityProp21Left} setting 
$y=v=o$ respectively $x=u=o$.
For the relayed uplink case $C(\eta,\nu)=R(\eta,o,\nu)$ we have to prove \eqref{InequalityProp21Left} only for the relaying component. In other words, we show that for all $x\in\rr^{-1}(u)$ and $t\in I_\de$
\begin{equation*}
\begin{split}
\nuu\text{-}\esssup_{y\in\LL}\Gamma(x_s,y_s,o,\nuu_s)\le\nuu^\rr\text{-}\esssup_{v\in\Pi_\delta}\Gamma(u_t,v_t,o,(1-\e)\nuu_t^\rr)
\end{split}
\end{equation*}
for all $t-\tfrac{\de T}{2}\le s<t+\tfrac{\de T}{2}$. Note that for all $x\in\rr^{-1}(u)$ and $t\in I_\de$
\begin{equation*}
\begin{split}
\nuu\text{-}\esssup_{y\in\LL}\Gamma(x_s,y_s,o,\nuu_s)=\nuu^\rr\text{-}\esssup_{v\in\Pi_\delta}\big(\nuu\text{-}\esssup_{y\in\rr^{-1}(v)}\Gamma(x_s,y_s,o,\nuu_s)\big)
\end{split}
\end{equation*}
and assume that this supremum is attained in $v\in \Pi_\delta$ where necessarily $\nuu^\rr(v)>0$. Then it suffices to find $\delta'=\delta'(\e)\in\mathbb{B}$ such that for all $\delta\in\mathbb{B}\cap(0,\delta')$, $x\in\rr^{-1}(u)$, $y\in\rr^{-1}(v)$ and $t\in I_\de$
\begin{equation*}
\begin{split}
\min\{D(x_s,y_s,\nuu_s),D(y_s,o,\nuu_s)\}\le\min\{D(u_t,v_t,(1-\e)\nuu_t^\rr),D(v_t,o,(1-\e)\nuu_t^\rr)\}
\end{split}
\end{equation*}
for all $t-\tfrac{\de T}{2}\le s<t+\tfrac{\de T}{2}$.
But this can be done using \eqref{Central2InequalityProp21Left}.

Similar, in the case of relayed downlink communication $C(\eta,\nu)=R(o,\eta,\nu)$, using the same argument as in the relayed uplink case, we need to show
\begin{equation*}
\begin{split}
\min\{D(o,y_s,\nuu_s),D(y_s,x_s,\nuu_s)\}\le\min\{D(o,v_t,(1-\e)\nuu_t^\rr),D(v_t,u_t,(1-\e)\nuu_t^\rr)\}
\end{split}
\end{equation*}
for all $x\in\rr^{-1}(u)$, $y\in\rr^{-1}(v)$, $t\in I_\de$, $t-\tfrac{\de T}{2}\le s<t+\tfrac{\de T}{2}$ and sufficiently small $\de$. But this is also true using \eqref{Central2InequalityProp21Left}. This finishes the proof.

\subsection{Proof of Proposition~\ref{LDP_Discr}}
The idea for the proof of Proposition~\ref{LDP_Discr} is to apply Varadhan's lemma~\cite[Lemmas 4.3.4, 4.3.6]{dz98} to a suitable functional on $L_\lambda^\rr$. More precisely,~\cite[Exercise 4.2.7]{dz98} implies that the independent random variables $\{L^\rr_\lambda(u)\}_{u\in \Pi_\delta}$ satisfy a large deviation principle with good rate function 
$$\{a_u\}_{u\in \Pi_\delta}\mapsto \sum_{u\in \Pi_\delta}h(a_u|\muu^\rr(u)).$$

\medskip\noindent
Let us start with the upper bound. In Remark \ref{USC}
we have seen that for every $u\in\Pi_\delta$ the map 
$$\nuu\mapsto \big( \tau_1(\bar R(u,o,\a\nuu)),\tau_2(\bar D(u,o,\a\nuu)),\tau_3(\bar R(o,u,\a\nuu)),\tau_4(\bar D(o,u,\a\nuu))\big)$$ 
is u.s.c. Hence, also the map $\nuu\mapsto(\a\nuu)[\ttau]$ is u.s.c. Finally, another application of Lemma~\ref{uscLem} shows that $F((\a\nuu)[\ttau])$ is u.s.c.~as a function of $\nuu$, so that the upper bound in Proposition~\ref{LDP_Discr} is an immediate consequence of Varadhan's lemma~\cite[Lemma 4.3.6]{dz98}.


\medskip\noindent
In contrast, the lower bound requires a substantial amount of additional work. 
Since the map $\nuu\mapsto\nuu[\ttau]$ is not l.s.c.~the proof of the lower bound in Proposition~\ref{LDP_Discr} is substantially more involved than the proof of the upper bound. Therefore, we first introduce a l.s.c.~approximation of the mapping $\nuu\mapsto\nuu[\ttau]$. As we will see, the cost of this approximation is negligible on the exponential scale. To be more precise, for the uplink, we introduce the approximating measure 
$$\nuu^{\ms{up}}[\tau,\e](A)=\sum_{u\in A}\nuu(u)\tau(\{R_{\e}(u_t,o,\nuu)\}_{t\in I_\delta}), \qquad A\subset \Pi_\delta,$$
where 
$${R}_{\varepsilon}(u_t,o,\nuu)=\max\big\{D(u_t,o,\nuu_{t}),\max_{v\in \Pi_{\delta}}\Gamma_{\e}(u_t,v_t,o,\nuu_{t})\big\},$$ 
is defined using $\Gamma_{\e}(u_t,v_t,o,\nuu_{t})=\min\{1,\e^{-1}\nuu_t(v_t)\}\Gamma(u_t,v_t,o,\nuu_{t})$. In particular, 
$$\Gamma_{\e}(u_t,v_t,o,\nuu_t)\le \one\{\nuu_{t}(v_t)>0\}\Gamma(u_t,v_t,o,\nuu_{t}),$$
 where equality holds if and only if $\nuu_{t}(v_t)\in\{0\}\cup[\e,\infty)$. Similarly, for the downlink we introduce the approximating empirical measures $\nuu^{\ms{do}}[\tau,\e]$ and $\nuu^{\ms{do-dir}}[\tau]$ and put
$$\nuu[\ttau,\e]=\big(\nuu^{\ms{up}}[\tau_1,\e],\nuu^{\ms{up-dir}}[\tau_2],\nuu^{\ms{do}}[\tau_3,\e],\nuu^{\ms{do-dir}}[\tau_4]\big).$$
The following result formalizes the approximation property under the event $E_{\e_0}\cap E'_{\e_0}$ from Section~\ref{Sprinkling}.

\begin{Lemma}
\label{tntnpLem5}
Let $0<\a_-<\a<2$ and $\tau_i:[0,\infty)^{I_\delta}\to[0,\infty)$, $i\in\{1,\ldots,4\}$ be decreasing measurable functions such that $\tau_i(\gamma)=0$  for every $i\in\{1,\ldots,4\}$ if $\gamma_t\ge c_+$ for every $t\in I_\delta$. Then, for every sufficiently small $\e_0>0$ there exists $\lambda_0=\lambda_0(\e_0)$ with the following properties. 
 If $\lambda\ge \lambda_0$, then, almost surely, for every $u\in\Pi_\delta$ and $t\in I_\delta$,
$$\one\{E_{\e_0}\cap E'_{\e_0}\}{R}(u_t,o,\a L_{\lambda,t}^\rr)\le{R}_{\e_0\a_-}(u_t,o,\a_-L_{\lambda',t}^\rr),$$
and
$$\one\{E_{\e_0}\cap E'_{\e_0}\}{R}(o,u_t,\a L_{\lambda,t}^\rr)\le{R}_{\e_0\a_-}(o,u_t,\a_-L_{\lambda',t}^\rr)$$
where $\lambda'=(1+2\e_0\k_\delta^{-1})\lambda$. In particular, 
$$\one\{E_{\e_0}\cap E'_{\e_0}\}(\a_-L^\rr_{\lambda'})[\ttau,{\e_0}\a_-]\le\a L^\rr_\lambda[\ttau].$$
\end{Lemma}
\begin{proof}
First, as the event $E'_{\e_0}$ occurs, we may apply part (v) of Lemma~\ref{gngnpCor} and deduce that $D(u_t,v_t,\a L^\rr_{\lambda,t})\le D(u_t,v_t,\a_-L_{\lambda',t}^\rr)$ holds for all $u,v\in \Pi_\delta$ and $t\in I_\delta$. Now, it suffices to show that under the event $E_{\e_0}\cap E'_{\e_0}$
$$\Gamma(u_t,v_t,w_t,\a L_{\lambda,t}^\rr)\le\min\{1,\e_0^{-1}L_{\lambda',t}^\rr(v_t)\}\Gamma(u_t,v_t,w_t,\a_-L_{\lambda',t}^\rr)$$
holds for all $u,v,w\in \Pi_\delta$ and  $t\in I_\delta$ with $L_{\lambda,t}^\rr(v_t)>0$. We claim that $L_{\lambda',t}^\rr(v_t)\ge{\varepsilon_0}$ under the event $E_{\e_0}$. Indeed, otherwise
${\e_0}> L_{\lambda',t}^\rr(v_t)=\sum_{u\in\Pi_\de:\, u_t=v_t}L_{\lambda'}^\rr(u)$ and thus $0=\sum_{u\in\Pi_\de:\, u_t=v_t}L_{\lambda}^\rr(u)=L_{\lambda,t}^\rr(v_t)$.
Now we conclude by applying the inequality for $D$.
\end{proof}

Next, we show that Lemma~\ref{tntnpLem5} implies closeness in the exponential scale.

\begin{Lemma}
\label{lbProp1}
Let $0<\a_-<\a<2$. Let $\tau_i:[0,\infty)^{I_\delta}\to[0,\infty)$, $i\in\{1,\ldots,4\}$ and $F:\mc{M}(\Pi_\delta)^4\to[-\infty,\infty)$ be measurable functions such that the $\tau_i$ are  decreasing and $F$ is increasing. Furthermore, assume that $\tau_i(\gamma)=0$ for every $i\in\{1,\ldots,4\}$ if $\g_t\ge c_+$ for every $t\in I_\delta$ and that $F$ maps the vector of zero measures to $-\infty$. Then, for every sufficiently small $\varepsilon_0\in(0,1)$ there exists $\lambda_0=\lambda_0(\varepsilon_0)$ 
such that for all $\lambda\ge \lambda_0$ 
$$\E\exp\big( \la F(\a L_{\lambda}^\rr[\ttau])\big)\ge\exp(-\sqrt{\varepsilon_0}\lambda)\E\exp \big(\la F((\a_-L^\rr_{\lambda'})[\ttau,{\e_0}\a_-])\big).$$
\end{Lemma}

\begin{proof}
First, Lemma~\ref{tntnpLem5} implies that 
\begin{align*}
	\E\Big(\exp\big( \la F(\a L_{\lambda}^\rr[\ttau])\big)\Big)\ge\E\Big(\one\{E_{\e_0}\cap E'_{\e_0}\}\exp \big(\la F((\a_-L^\rr_{\lambda'})[\ttau,{\e_0}\a_-])\big)\Big).
\end{align*}
Note that if $L^\rr_{\lambda'}(\Pi_\delta)<\b_o/2$, then $(\a_-L^\rr_{\lambda'})[\ttau,{\e_0}\a_-]=0$. 
Hence, using the assumption that $F$ maps the vector of zero measures to $-\infty$, we may apply Lemma~\ref{condProbLem} to deduce that 
\begin{align*}
\E\exp\big( \la F(\a L^\rr_{\lambda}[\ttau])\big)&\ge \E\Big(\P(E_{\e_0}\cap E'_{\e_0}|X^{\lambda'}_\delta)\exp \big(\la F((\a_-L^\rr_{\lambda'})[\ttau,{\e_0}\a_-])\big)\Big)\\
&\ge\exp(-\sqrt{\varepsilon_0}\lambda)\E\Big(\one\{L_{\lambda'}^\rr(\Pi_\delta)\ge \b_o/2\}\exp \big(\la F((\a_-L^\rr_{\lambda'})[\ttau,{\e_0}\a_-])\big)\Big)\\
&=\exp(-\sqrt{\varepsilon_0}\lambda)\E\exp \big(\la F((\a_-L^\rr_{\lambda'})[\ttau,{\e_0}\a_-])\big),
\end{align*}
as required.
\end{proof}

Now, we can proceed with the proof of the lower bound of Proposition~\ref{LDP_Discr}. 
As a first step, we note that the map $\nuu\mapsto\nuu[\ttau,\e_0]$ is continuous. Hence for $0<\a_-<\a$, combining Lemma~\ref{lbProp1} with Varadhan's lemma shows that
\begin{align*}
\liminf_{\lambda\to\infty}\tfrac1\lambda\log\E\exp\big( \la F((\a L^\rr_{\lambda})[\ttau])\big)&\ge-\sqrt{\e_0}+ \liminf_{\lambda\to\infty}\tfrac1\lambda\log\E\exp\big( \la F((\a_-L^\rr_{\lambda'})[\ttau,{\e_0}\a_-])\big)\\
&\ge -\sqrt{\e_0}-(1+2\e_0\k_\delta^{-1})\hspace{-0.15cm}\inf_{\nuu\in\mc{M}(\Pi_\delta)}\hspace{-0.1cm}\{h(\nuu|\muu^\rr)-F((\a_-\nuu)[\ttau,{\e_0}\a_-])\}.
\end{align*}
Moreover, since $\tau_1$ is decreasing, we have 
$$\tau_1(\{R(u_t,o,\a_-\nuu_t)\}_{t\in I_\delta})\le\tau_1(\{R_{\e_0\a_-}(u_t,o,\a_-\nuu_t)\}_{t\in I_\delta})$$ 
for all $u\in\Pi_\delta$ and $\nuu\in\mc{M}(\Pi_\de)$. 
Similarly for the other communication cases. Hence, 
$$(\a_-\nuu)[\ttau,\e_0\a_-]\ge(\a_-\nuu)[\ttau]$$
and sending $\e_0$ to zero  yields
$$\liminf_{\lambda\to\infty}\tfrac1\lambda\log\E\exp \big(\la F(L^\rr_{\lambda}[\ttau])\big)\ge-\inf_{\nuu\in\mc{M}(\Pi_\delta)}\{h(\nuu|\muu^\rr)-F((\a_-\nuu)[\ttau])\}.$$
Therefore, it remains to verify that
$$\inf_{\nuu\in\mc{M}(\Pi_\delta)}\{h(\nuu|\muu^\rr)-F(\nuu[\ttau])\}\ge \limsup_{\a_-\uparrow\a}\inf_{\nuu\in\mc{M}(\Pi_\delta)}\{h(\a_-^{-1}\nuu|\muu^\rr)-F(\nuu[\ttau])\}.$$
In order to prove this claim, let $\nuu\in\mc{M}(\Pi_\delta)$ be arbitrary. If $\nuu$ is not absolutely continuous with respect to $\muu^\rr$, then the left-hand side is infinite and there is nothing to show. Otherwise, Lemma~\ref{ent0Lem} shows that $\lim_{\a_-\uparrow\a}h(\a_-^{-1}\nuu|\muu^\rr)=h(\alpha\nuu|\muu^\rr)$, as required.

\section{Proof of Theorem~\ref{LDP}}
\label{LDPSec}

After having established Propositions~\ref{translUpThm} and~\ref{LDP_Discr}, the proof of Theorem~\ref{LDP} is reduced to the following result on the behavior of the rate functions in Proposition~\ref{LDP_Discr} as $\delta$ tends to zero.
\begin{Lemma}
\label{rateFunLim}
Let $F:\mc{M}(\BBB)^4\to[-\infty,\infty)$ and $\tau_i:\BB\to[0,\infty)$, $i\in\{1,\ldots,4\}$ be measurable functions that are respectively increasing and decreasing. Furthermore, assume that $F$ is $\delta_0$-discretized for some $\delta_0\in\mathbb{B}$ and bounded from above. Then,
$$\lim_{\e\to0}\limsup_{\delta\to0}\inf_{\nuu\in\mathcal{M}(\Pi_\delta)}\big\{h(\nuu|\muu^\rr)-F(((1-\e)\nuu)^\imath[\ttau])\big\}\le\inf_{\nuu\in\mc{M}(\LL)}\big\{h(\nuu|\muu)-F^{}(\nuu[\ttau])\big\},$$
and 
$$\lim_{\e\to0}\liminf_{\delta\to0}\inf_{\nuu\in\mathcal{M}(\Pi_\delta)}\big\{h(\nuu|\muu^\rr)-F(((1+\e)\nuu)^\imath[\ttau])\big\}\ge\inf_{\nuu\in\mc{M}(\LL)}\big\{h(\nuu|\muu)-F^{}(\nuu[\ttau])\big\},$$
where in the limits it is assumed that $\delta\in\mathbb{B}$.
\end{Lemma}

Before we provide a proof of Lemma~\ref{rateFunLim}, let us show how it can be used to complete the proof of Theorem~\ref{LDP}.
\begin{proof}[Proof of Theorem~\ref{LDP}]
We only provide a proof for the lower bound, the proof for the upper bound is analogous. Let  $\e\in(0,1)$ be arbitrary. Then, Propositions~\ref{translUpThm} and~\ref{LDP_Discr} show that for all sufficiently small $\delta\in\mathbb{B}$ 
\begin{align*}
\liminf_{\lambda\to\infty}\tfrac1\lambda \log\E\exp\big(\la F(L_\lambda[\ttau])\big)&=\liminf_{\lambda\to\infty}\tfrac1\lambda \log\E\exp\big(\la F(((L_\lambda[\ttau])^\rr)^\imath)\big)\\
&\ge\liminf_{\lambda\to\infty}\tfrac1\lambda \log\E\exp \big(\la F((((1-\e)L_\lambda^\rr)[\ttau\circ\imath])^\imath)\big)\\
&\ge-\inf_{\nuu\in\mathcal{M}(\Pi_\delta)}\big\{h(\nuu|\muu^\rr)-F((((1-\e)\nuu)[\ttau\circ \imath])^\imath)\big\}\\
&=-\inf_{\nuu\in\mathcal{M}(\Pi_\delta)}\big\{h(\nuu|\muu^\rr)-F(((1-\e)\nuu)^\imath[\ttau])\big\}.
\end{align*}
Hence, applying Lemma~\ref{rateFunLim} yields that 
\begin{align*}
	\liminf_{\lambda\to\infty}\tfrac1\lambda\log\E\exp(\la F(L_\lambda[\ttau]))\ge-\inf_{\nuu\in\mathcal{M}(\LL)}\big\{h(\nuu|\muu)-F(\nuu[\ttau])\big\},
\end{align*}
as required.
\end{proof}
Now, we prove Lemma~\ref{rateFunLim}.
\begin{proof}[Proof of Lemma~\ref{rateFunLim}] First, we consider the upper bound.
Let $\e_0\in(0,1)$ and $\nuu_0\in\mc{M}(\LL)$ be arbitrary. Then, we need to show that 
$$\limsup_{\delta\to0}\inf_{\nuu\in\mc{M}(\Pi_\delta)}\big\{h(\nuu|\muu^\rr)-F(((1-\e)\nuu)^\imath[\ttau])\big\}\le \e_0+h(\nuu_0|\muu)-F(\nuu_0[\ttau]),$$
holds provided that $\e\in(0,1)$ is sufficiently small. Since $F$ is bounded from above, we may focus on the case where $\nuu_0$ is absolutely continuous with respect to $\muu$. 

First, Proposition~\ref{translUpThm} shows that if $\delta$ is sufficiently small, then $(\nuu_0[\ttau])^\rr\le((1+\e)\nuu_0^\rr)[\ttau\circ\imath]$. In particular,
$$F(\nuu_0[\ttau])=F(((\nuu_0[\ttau])^\rr)^\imath)\le F(((1+\e)\nuu_0^\rr)^\imath[\ttau]).$$
Hence, putting $1+\e'=(1+\e)(1-\e)^{-1}$, it suffices to show that
$$\limsup_{\delta\to0}h((1+\e')\nuu_0^\rr|\muu^\rr)\le\e_0+h(\nuu_0|\muu)$$
holds for all sufficiently small $\e$. First, note that Jensen's inequality yields $h(\nuu_0|\muu)\ge h(\nuu_0^\rr|\muu^\rr)$. Hence, by Corollary~\ref{ent0Cor},
\begin{align*}
h((1+\e')\nuu_0^\rr|\muu^\rr)-h(\nuu_0^\rr|\muu^\rr)\le3\e'h(\nuu_0^\rr|\muu^\rr)+3\e'\muu^\rr(\Pi_\delta)\le3\e' h(\nuu_0|\muu)+3\e'\muu(\LL).
\end{align*}
Since this upper bound tends to zero as $\e$ tends to zero, we conclude the proof.

Next, we consider the lower bound. Let $\e_0\in(0,1)$ be arbitrary. Then, we have to show that 
$$\liminf_{\delta\to0}\inf_{\nuu\in\mc{M}(\Pi_\delta)}\big\{h(\nuu|\muu^\rr)-F(((1+\e)\nuu)^\imath[\ttau])\big\}\ge -\e_0+\inf_{\nuu\in\mc{M}(\LL)}\big\{h(\nuu|\muu)-F(\nuu[\ttau])\big\}.$$
holds provided that $\e\in(0,1)$ is sufficiently small. 
First, for any $\e\in(0,1)$ we choose a suitable sequence $\{\delta_k\big\}_{k\ge1}$ in $\mathbb{B}$ such that $\lim_{k\to\infty}\delta_k=0$ and such that the $\liminf_{\delta\to0}$ above is replaced by $\lim_{\delta_k\to0}$. 
Moreover, for $\e\in(0,1)$ and $k\ge1$ choose $\nuu_{k,\e}\in\mc{M}(\Pi_{\delta_k})$ such that
$$h(\nuu_{k,\e}|\muu^{\rr_{\delta_k}})-F(((1+\e)\nuu_{k,\e})^\imath[\ttau])\le\e_0/2+\inf_{\nuu\in\mc{M}(\Pi_\delta)}\big\{h(\nuu|\muu^{\rr_{\delta_k}})-F(((1+\e)\nuu)^\imath[\ttau])\big\}.$$
Hence, it remains to show that 
$$\liminf_{\e\to0}\liminf_{k\to\infty}h(\nuu_{k,\e}|\muu^{\rr_{\delta_k}})-F(((1+\e)\nuu_{k,\e})^\imath[\ttau])\ge\inf_{\nuu\in\mathcal{M}(\LL)}\big\{h(\nuu|\muu)-F^{}(\nuu[\ttau])\big\},$$
In particular, we may assume that $\nuu_{k,\e}$ is absolutely continuous with respect to $\muu^{\rr_{\delta_k}}$. 
Then, we define $\nuu_{k,\e}'\in\mc{M}(\LL)$ by 
$$\nuu_{k,\e}'(\cdot)=(1+2\e)\sum_{u\in \Pi_{\delta_k}}\frac{\nuu_{k,\e}(u)}{\muu^{\rr_{\delta_k}}(u)}\muu(\rr_{\delta_k}^{-1}(u)\cap\cdot),$$
so that $h(\nuu_{k,\e}'|\muu)=h((\nuu_{k,\e}')^{\rr_{\delta_k}}|\muu^{\rr_{\delta_k}})$. 
Moreover, Proposition~\ref{translUpThm} implies that 
$$(\nuu_{k,\e}'[\ttau])^{\rr_{\delta_k}}\ge((1-\e'')(1+2\e)\nuu_{k,\e})[\ttau\circ\imath]=((1+\e)\nuu_{k,\e})[\ttau\circ\imath]$$
for all sufficiently small $\delta_k\in\B$ where $1-\e''=(1+\e)(1+2\e)^{-1}$. Hence, by Corollary~\ref{ent0Cor},
\begin{align}
\label{LDEntBoundEq}
h((\nuu_{k,\e}')^{\rr_{\delta_k}}|\muu^{\rr_{\delta_k}})-h(\nuu_{k,\e}|\muu^{\rr_{\delta_k}})\le6\e h(\nuu_{k,\e}|\muu^{\rr_{\delta_k}})+6\e\muu(\LL).
\end{align}
The boundedness of $F$ implies that if
$$\liminf_{\e\to0}\liminf_{k\to\infty}h(\nuu_{k,\e}|\muu^{\rr_{\delta_k}})=\infty,$$
then there is nothing to show. Otherwise,~\eqref{LDEntBoundEq} gives
$$\liminf_{\e\to0}\liminf_{k\to\infty}h((\nuu_{k,\e}')^{\rr_{\delta_k}}|\muu^{\rr_{\delta_k}})-h(\nuu_{k,\e}|\muu^{\rr_{\delta_k}})\le0,$$
as required.
\end{proof}

\section{Proof of Corollaries~\ref{Cor1} and~\ref{Cor2}}
\label{cor1Sec}

First, defining the maps $F_{\bb}:\mc{M}(\BBB)^4\to[-\infty,\infty)$ and $\ttau_{\aa,\cc}:\BB^4\to[0,\infty)^4$  as in~\eqref{ExF} and~\eqref{ExF2}, we see that the maps $\nuu\mapsto F_\bb(\nuu^\imath)$ and $\ttau_{\aa,\cc}\circ\imath$ are l.s.c.~on $\mc{M}(\Pi_\delta)^4$ and $([0,\infty)^{I_\de})^4$, respectively. 
Hence, by Theorem~\ref{LDP}, only the upper bound requires a proof. In the following, we restrict to the case $\bb\ge0$. This is no substantial loss of generality, as negative coordinates of $\bb$ translate into putting no constraints on the corresponding component of $L_\lambda[\ttau_{\aa,\cc}]$. 

We first derive the upper bound in Corollary~\ref{Cor1} in the discretized model. As before, we fix $\delta\in\mathbb{B}$ such that $\kappa_{\delta}\le1$.
\begin{Proposition}
\label{translLowThm}
Let $0<\a<2$, $\aa\in[0,{\bf T})$, $\bb\ge0$ and $\cc\in[0,\cc_+)$. Then,
$$\limsup_{\lambda\to\infty}\tfrac1\lambda\log\P((\a L_\lambda^\rr)[\ttau_{\aa,\cc}\circ\imath](\Pi_\delta)>\bb)\le-\inf_{\substack{\nuu\in\mc{M}(\Pi_\delta)\\ (\a\nuu)[\ttau_{\aa,\cc}\circ\imath](\Pi_\delta)>\bb}}h(\nuu|\muu^\rr).$$
\end{Proposition}
The lack of upper semicontinuity in the maps $\nuu\mapsto F_\bb(\nuu^\imath)$ and $\ttau_{\aa,\cc}\circ\imath$ prevents us from applying Proposition~\ref{LDP_Discr} directly. However, if we define $\tau^{\ms{usc}}_{a,c}:\,[0,\infty)^{I_\delta}\to[0,\infty)$ by 
\begin{align}
\g\mapsto
\begin{cases}
1& \text{if } \int_0^T\one\{\imath(\g)_t\le c\}\d t>a,\\
0& \text{otherwise,}
\end{cases}
\end{align}
then $\tau^{\ms{usc}}_{a,c}$ is u.s.c. For $\aa\in[0,{\bf T})$ and $\cc\in[0,\cc_+)$ we put $\ttau^{\ms{usc}}_{\aa,\cc}=(\tau^{\ms{usc}}_{a_i,c_i})_{i\in\{1,\ldots,4\}}$.

If we knew that $\bb>0$, then $L_\lambda^\rr[\ttau_{\aa,\cc}^{\ms{usc}}](\Pi_\delta)\ge\bb$ would be a useful u.s.c.~approximation of the considered event. However, to deal with the general case where certain entries of $\bb$ may be zero, it will be convenient to introduce further quantities describing the worst QoS that is experienced by any user in the system for a period of time of length larger than $a_i$. To be more precise, for $\xi\in W_\delta$ and  $\nu\in\mc{M}(W_\delta)$ we put 
$$\boldsymbol\Phi(\xi,\nu)=(R(\xi,o,\nu),D(\xi,o,\nu),R(o,\xi,\nu),D(o,\xi,\nu))$$
and note that for fixed $\xi\in W_\delta$, $\nu\mapsto\boldsymbol\Phi(\xi,\nu)$ is l.s.c, see Lemma~\ref{upSemContLem}. Here the discontinuities come from the effect that sites can become unavailable as possible relay locations if the limiting number of users at certain sites is zero. 
Further, for $\cc\in[0,\cc_+)$, $u\in\Pi_\delta$ and $\nuu\in\mc{M}(\Pi_\delta)$, we define 
$$\boldsymbol\Phi(\cc,u,\nuu)=\Big(\sum\nolimits_{t\in I_\delta}\delta T\one\{\pi_i(\boldsymbol\Phi(u_t,\nuu_t))\le c_i\}\Big)_{i\in\{1,\ldots,4\}}$$
as the total amount of time that a user $u$ experiences bad QoS of at most $c_i$.
Finally, we define
$$\boldsymbol\Phi(\cc,\nuu)=\max_{u\in \Pi_\delta:\, \nuu(u)>0}\boldsymbol\Phi(\cc,u,\nuu),$$
as the maximum amount of time that a user from $\nu$ experiences bad QoS of at most $c_i$. 
For instance, the event $\{\nuu^{\ms{up}}[\ttau_{\aa,\cc}^{\ms{usc}}](\Pi_\delta)>0\}$
can now be rewritten as $\{\pi_1(\boldsymbol\Phi(\cc,\nuu))>a_1\}$ and analogous relationships are true for the other three components.

Unfortunately, $\boldsymbol\Phi(\cc,\nuu)$ does not satisfy any semicontinuity properties. For example discontinuities can come from the effect
that users along trajectories with bad QoS become irrelevant if the number of these users tends to zero.
Let us therefore introduce the approximations
$$\boldsymbol\Phi_{\e}(\cc,\nuu)=\max_{\substack{u\in \Pi_\delta}}\big\{\boldsymbol\Phi(\cc,u,\nuu)\min\{1,\e^{-1}\nuu(u)\}\big\}.$$
In particular, $\boldsymbol\Phi(\cc,\nuu)\ge\boldsymbol\Phi_{\varepsilon}(\cc,\nuu)$.

In the following, for $\varepsilon>0$, $\aa\in[0,{\bf T})$, $\bb\ge0$,  $\cc\in[0,\cc_+)$, $i\in\{1,\ldots,4\}$  we define
$$C_i(\aa,\bb,\cc,\e)=
\begin{cases}
\{\nuu\in\mc{M}(\Pi_\delta):\, \pi_i(\boldsymbol\Phi_\e(\cc,\nuu))>a_i\}&\text{if $b_i=0$},\\ \{\nuu\in\mc{M}(\Pi_\delta):\, \pi_i(\nuu[\t^{\ms{usc}}_{\aa,\cc}])\ge b_i\}&\text{if $b_i>0$}.
\end{cases}$$
Moreover, we put 
$${\bf C}(\aa, \bb,\cc,\e)=\bigcap_{i=1}^4C_i(\aa,\bb,\cc,\e).$$
Note that ${\bf C}(\aa,\bb,\cc,\e)$ is a closed set, since the maps $\nuu\mapsto \nuu[\ttau^{\ms{usc}}_{\aa,\cc}]$ and $\nuu\mapsto \boldsymbol\Phi_\e(\cc,\nuu)$ are u.s.c. 
Note that by Lemma~\ref{gngnpCor} parts (ii) and (iv), for every $\e>0$ and $\a _+>\a >0$ we have an inclusion
$$\{\a  L^\rr_\lambda\in {\bf C}(\aa,(1+\e)\bb,\cc,\e)\}\subset\{(\a _+L^\rr_\lambda)[\ttau^{\ms{}}_{\aa,\cc}\circ\imath](\Pi_\delta)>\bb\}.$$
Now we show that under the event $E^*_{\e_0}$
introduced in Section \ref{prelSec}, for $\a _+>\a >0$ the inclusion 
$$\{\a  L^\rr_\lambda\in {\bf C}(\aa,\sqrt{\a _+\a ^{-1}}\bb,\cc,\a _+\e_0)\}\subset\{(\a _+L^\rr_\lambda)[\ttau^{\ms{}}_{\aa,\cc}\circ\imath](\Pi_\delta)>\bb\}$$
 is not far from being an equality.
\begin{Lemma}
\label{tntnpLem}
Let $\a _+>\a >0$, $\aa\in[0,{\bf T})$, $\bb\ge0$ and $\cc\in[0,\cc_+)$ be arbitrary. Then, for every sufficiently small $\e_0\in(0,1)$ there exists $\lambda_0=\lambda_0(\e_0)$ with the following properties. If $\lambda\ge\lambda_0$, then
$$E^*_{\e_0}\cap\{(\a  L^\rr_\lambda)[\ttau_{\aa,\cc}\circ\imath](\Pi_\delta)>\bb\}\subset\{\a _+L^\rr_{\lambda'}\in{\bf C}(\aa,\sqrt{\a _+\a ^{-1}}\bb,\cc,\a _+\e_0)\},$$
where $\lambda'=(1+2\e_0\kappa_\delta^{-1})\lambda.$
\end{Lemma}
\begin{proof}
First, recall that the event $E^*_{\e_0}$ guarantees that by passing from $\lambda$ to $\lambda'$ users can only be added along occupied trajectories. Hence, under the event $E^*_{\e_0}$ parts (i) and (iii) of Lemma~\ref{gngnpCor} give that
$$\bar D(u,v,\a _+L^\rr_{\lambda'})\le\bar D(u,v,\a  L_\lambda^\rr)\qquad\text{and}\qquad\bar R(u,v,\a _+L^\rr_{\lambda'})\le\bar R(u,v,\a  L_\lambda^\rr)$$
for all $u,v\in\Pi_\delta$ provided that $\e_0$ is sufficiently small. 
Therefore, under the event $E^*_{\e_0}$, $\pi_i((\a  L^\rr_\lambda)[\ttau_{\aa,\cc}\circ\imath](\Pi_\delta))>b_i$ implies that 
$$\pi_i((\a _+L^\rr_{\lambda'})[\ttau^{\ms{usc}}_{\aa,\cc}](\Pi_\delta))\ge\frac{\a _+\lambda}{\a \lambda'}b_i\ge \sqrt{\a _+\a ^{-1}}b_i.$$ 
In particular, $\a _+L^\rr_{\lambda'}\in C_i(\aa,\sqrt{\a _+\a ^{-1}}\bb,\cc,\a _+\e_0)$ if $b_i>0$.

For the case $b_i=0$, let $u\in \Pi_{\delta}$ with $L^\rr_\lambda(u)>0$ be arbitrary.
Since the event $E^*_{\e_0}$ occurs, we have $L_{\lambda'}^\rr(u)\ge\e_0$ and therefore
$$\min\{1,(\a _+\e_0)^{-1}\a _+L^\rr_{\lambda'}(u)\}=1.$$
Moreover, by parts (i) and (iii) of Lemma~\ref{gngnpCor}, we have that $\boldsymbol\Phi(\cc,\a _+L^\rr_{\lambda'})\ge\boldsymbol\Phi(\cc,\a L^\rr_{\lambda})$. Therefore,
$$\pi_i(\boldsymbol\Phi_{\a _+\varepsilon_0}(\cc,\a _+L^\rr_{\lambda'}))=\pi_i(\boldsymbol\Phi(\cc,\a _+L^\rr_{\lambda'}))\ge\pi_i(\boldsymbol\Phi(\cc,\a L^\rr_{\lambda}))>a_i,$$
i.e., $\a _+L^\rr_{\lambda'}\in C_i(\aa,\sqrt{\a _+\a ^{-1}}\bb,\cc,\a _+\e_0)$.
\end{proof}

Now, we can conclude the proof of Proposition~\ref{translLowThm}.
\begin{proof}[Proof of Proposition~\ref{translLowThm}]
First, by Lemmas~\ref{condProbLem2} and~\ref{tntnpLem},
\begin{align*}
\P((\a  L_\lambda^\rr)[\ttau_{\aa,\cc}](\Pi_\delta)>\bb)&\le\exp(\sqrt{\varepsilon_0}\lambda)\P(E^*_{\e_0}\cap\{\a _+L^\rr_{\lambda'}\in {\bf C}(\aa,\sqrt{\a _+\a ^{-1}}\bb,\cc,\a_+{\e_0})\})\\
&\le\exp(\sqrt{\varepsilon_0}\lambda)\P(\a _+L^\rr_{\lambda'}\in {\bf C}(\aa,\sqrt{\a _+\a ^{-1}}\bb,\cc,\a _+{\e_0})).
\end{align*}
In particular, combining the large deviation principle for $L^\rr_\lambda$ with the observation preceding Lemma~\ref{tntnpLem} yields
\begin{align*}
\limsup_{\lambda\to\infty}\tfrac1\lambda\log\P((\a  L^\rr_\lambda)[\ttau_{\aa,\cc}](\Pi_\delta)>\bb)&\le\sqrt\e_0-(1+2\e_0\kappa_{\delta}^{-1})\inf_{\substack{\nuu\in\mc{M}(\Pi_\delta)\\ (\a _+\nuu)\in{\bf C}(\aa,\sqrt{\a _+\a ^{-1}}\bb,\cc,\a _+{\e_0})}}h(\nuu|\muu^\rr)\\
&\le\sqrt\e_0-(1+2\e_0\kappa_{\delta}^{-1})\inf_{\substack{\nuu\in\mc{M}(\Pi_\delta)\\ (\a _+^2\a ^{-1}\nuu)[\ttau_{\aa,\cc}\circ\imath](\Pi_\delta)>\bb}}h(\nuu|\muu^\rr).
\end{align*}
Hence, after sending $\e_0$ to zero,
\begin{align*}
\limsup_{\lambda\to\infty}\tfrac1\lambda\log\P(L^\rr_\lambda[\ttau_{\aa,\cc}\circ\imath](\Pi_\delta)>\bb)\le-\inf_{\substack{\nuu\in\mc{M}(\Pi_\delta)\\ (\a _+^2\a ^{-1}\nuu)[\ttau_{\aa,\cc}\circ\imath](\Pi_\delta)>\bb}}h(\nuu|\muu^\rr).
\end{align*}
Furthermore, by Corollary~\ref{ent0Cor},
\begin{align*}
-\inf_{\substack{\nuu\in\mc{M}(\Pi_\delta)\\ (\a \nuu)[\ttau_{\aa,\cc}\circ\imath](\Pi_\delta)>\bb}}h(\a ^2\a _+^{-2}\nuu|\muu^\rr)&\le-(1-3\e')\inf_{\substack{\nuu\in\mc{M}(\Pi_\delta)\\ (\a \nuu)[\ttau_{\aa,\cc}\circ\imath](\Pi_\delta)>\bb}}h(\nuu|\muu^\rr)+3\e'\muu^\rr(\Pi_\delta),
\end{align*}
where $\e'>0$ is chosen such that $1-\e'=\a ^2\a _+^{-2}$.
Sending $\a _+$ to $\a $ completes the proof.
\end{proof}

Next, we can conclude the proof of Corollary~\ref{Cor1}.
\begin{proof}[Proof of Corollary~\ref{Cor1}]
Let $\e\in(0,1)$ be arbitrary. First, recall that Proposition~\ref{translUpThm} yields
$$\P(L_\lambda[\ttau_{\aa,\cc}](\LL)>\bb)\le\P(((1+\e)L^\rr_\lambda)[\ttau_{\aa,\cc}\circ\imath](\Pi_\delta)>\bb)$$
for all sufficiently small $\delta\in\mathbb{B}$. Hence, Proposition~\ref{translLowThm} implies that
$$\limsup_{\lambda\to\infty}\tfrac1\lambda\log\P(L_\lambda[\ttau_{\aa,\cc}](\LL)>\bb)\le -\inf_{\substack{\nuu\in\mc{M}(\Pi_\delta)\\ ((1+\e)\nuu)[\ttau_{\aa,\cc}\circ\imath](\Pi_\delta)>\bb}}h(\nuu|\muu^\rr).$$
Now, Lemma~\ref{rateFunLim} gives that
$$\lim_{\e\to0}\liminf_{\delta\to0}\inf_{\substack{\nuu\in\mc{M}(\Pi_\delta)\\ ((1+\e)\nuu)[\ttau_{\aa,\cc}\circ\imath](\Pi_\delta)>\bb}} h(\nuu|\muu^\rr)\ge\inf_{\substack{\nuu\in\mc{M}(\LL)\\ \nuu[\ttau_{\aa,\cc}](\LL)>\bb}}h(\nuu|\muu),$$
as required.
\end{proof}

Finally, we prove Corollary~\ref{Cor2}. 
\begin{proof}[Proof of Corollary~\ref{Cor2}]
First, by Lemma~\ref{gngnpCor} parts (ii) and (iv) and Proposition~\ref{translUpThm},  
$$\pi_i(((1+\tfrac\e2)\muu^\rr)[\ttau_{\aa,\cc}^{\ms{usc}}](\Pi_\delta))\le\pi_i(((1+\tfrac34\e)\muu^\rr)[\ttau_{\aa,\cc}\circ\imath](\Pi_\delta))\le b_i$$ 
holds for all sufficiently small $\delta\in\B$. Moreover, again by Proposition~\ref{translUpThm}, it suffices to show that 
\begin{align}
\label{cor2Eq0}
\limsup_{\lambda\to\infty}\tfrac1\lambda\log\P(((1+\tfrac\e3)L^\rr_\lambda)[\ttau_{\aa,\cc}^{\ms{usc}}](\Pi_\delta)>\bb)<0
\end{align}
holds for all sufficiently small $\delta\in\B$. Next, Proposition~\ref{translLowThm} reduces~\eqref{cor2Eq0} to the verification of 
$$\inf_{\substack{\nuu\in\mc{M}(\Pi_\delta)\\ ((1+\tfrac\e3)\nuu)[\ttau_{\aa,\cc}^{\ms{usc}}](\Pi_\delta)>\bb}}h(\nuu|\muu^\rr)>0.$$
In order to derive a contradiction, we assume that there exist $\nuu_k\in\mc{M}(\Pi_\delta)$ such that $((1+\tfrac\e3)\nuu_k)[\ttau_{\aa,\cc}^{\ms{usc}}](\Pi_\delta)>\bb$ and $\lim_{k\to\infty}h(\nuu_k|\muu^\rr)=0$.   
In particular, after passing to a subsequence, the lower semicontinuity of $h(\cdot|\muu^\rr)$ implies that the measures $\nuu_k$ converge weakly to $\muu^\rr$. Hence, the upper semicontinuity of the function $\nuu\mapsto\nuu[\ttau_{\aa,\cc}^{\ms{usc}}]$ implies that $((1+\tfrac\e3)\muu^\rr)[\ttau_{\aa,\cc}^{\ms{usc}}](\Pi_\delta)\ge \bb$.
If $b_i>0$, then this together with parts (i) and (iii) of Lemma~\ref{gngnpCor} implies that $\pi_i(((1+\tfrac\e2)\muu^\rr)[\ttau_{\aa,\cc}^{\ms{usc}}](\Pi_\delta))> b_i$. But this is a contradiction to our assumption $\pi_i(((1+\tfrac\e2)\muu^\rr)[\ttau_{\aa,\cc}^{\ms{usc}}](\Pi_\delta))\le b_i$. On the other hand, if $b_i=0$, then we apply the above argument with the u.s.c.~function $\nuu\mapsto\Phi_{\kappa_\delta}(\cc,(1+\tfrac\e3)\nuu)$. More precisely, since $\Phi_{\k_\delta}(\cc,\cdot)$ takes values in a discrete set, and since $\Phi_{\k_\delta}(\cc,(1+\tfrac\e3)\nuu_k)>a_i$, we conclude that also $\Phi_{\k_\delta}(\cc,(1+\tfrac\e3)\muu^\delta)>a_i$. In particular,
\begin{align*}
 \pi_i(\boldsymbol\Phi(\cc,(1+\tfrac\e2)\muu^\rr))\ge\pi_i(\boldsymbol\Phi(\cc,(1+\tfrac\e3)\muu^\rr))=\pi_i(\boldsymbol\Phi_{\kappa_\delta}(\cc,(1+\tfrac\e3)\muu^\rr))>a_i,
\end{align*}
where we used parts (i) and (iii) of Lemma~\ref{gngnpCor} in the first inequality. Hence we obtain a contradiction to the assumption $\pi_i(\boldsymbol\Phi((1+\tfrac\e2)\muu^\rr))\le a_i$.
\end{proof}

\section{Simulation results}
\label{simSec}
In this section we provide some simulation results complementing the large-deviation analysis developed above. We restrict ourselves to a setting without mobility, where the state space of the point processes is $W$ rather then $\LL$. Already in this static situation a number of surprising effects, for example symmetry breaking, can be observed. 
In this section, we consider the specific frustration event 
$$E_{\lambda,\cc,\bb}=\{L_{\lambda}[\ttau_{0,\cc}](W)>\bb\},$$
also considered in the Corollaries \ref{Cor1} and \ref{Cor2}, where the proportion of users with a bad QoS is unexpectedly high. Recall that in Theorem~\ref{LDP} four different types of communication are considered, direct uplink communication, direct downlink communication, relayed uplink communication and relayed downlink communication. As we will see in the numerical analysis, for the different types of communication, different effects can lead to a configuration being an element of $E_{\lambda,\cc,\bb}$. Since we only consider a static scenario, we write $\ttau_{\cc}=\ttau_{0,\cc}$ in the following.

Asymptotically, these configurations are characterized by the minimizers of the rate function presented in Corollary \ref{Cor1}. 
For a network operator it is interesting to identify the reasons or bottlenecks behind bad connection quality, so that specific action may be taken to reduce these effects. For a large number of users, i.e.~for large parameter $\lambda$, the minimizers can be used to obtain qualitative information about the behavior of the system in such cases. More specifically, the set of minimizers represents the typical user distributions in the frustration event.

Most prominently we can observe a certain breaking of rotational symmetry in all cases except for the direct uplink communication where the interference is only measured at the base station at the origin. 
This symmetry breaking is to be understood in the following sense. As is true in general, the set of minimizers of the rate function must exhibit the same symmetries as the underlying system. In particular, if the a priori density $\mu$ is rotationally invariant, this must be true for the set of minimizers representing the typical user distribution in the frustration event. Symmetry breaking here means the existence of at least one element in the set of minimizers which is rotationally non-symmetric.

\subsection{Direct uplink communication}
We assume that the a priori measure is given by the restriction of the Lebesgue measure $\mathcal{H}_2$ on the disk of radius $r$ centered at the origin. That is, 
$$\mu(\cdot)=\mathcal{H}_2(\cdot\cap B_r(o)).$$
First, let us denote the minimum direct-communication QoS $c_0$ of users distributed according to a Poisson point process with intensity measure $\lambda\mu$ in the high-density limit when $\lambda$ tends to infinity. 
This minimum quality of service is attained at the boundary of the disk and can be computed as 
\begin{align*}
c_0=\frac{\ell(r)}{\int_{B_r(o)}\ell(|x|)\d x}.
\end{align*}

Now, we consider the frustration event 
$$E_{\lambda,c,b}^{\ms{up-dir}}=\{L_{\lambda}^{\ms{up-dir}}[\tau_c](B_r(o))>b\},$$
which describes the existence of at least $\lambda\, b$ users in $B_r(o)$ that experience a connection quality that is less than $c$. Here the empirical measure of frustrated users for the case of direct communication is given by
$$L_{\lambda}^{\ms{up-dir}}[\tau_c]=\frac1\lambda\sum_{X_j\in X^{\lambda}}\one\{\SIR(X_j,o,L_\lambda)< c\}\delta_{X_j}.$$ 
According to Corollary~\ref{Cor1} the probability for the event $E_{\lambda,c,b}^{\ms{up-dir}}$ is exponentially decaying at a rate $\lambda J^{\ms{up-dir}}(c,b)$ where
$$J^{\ms{up-dir}}(c,b)=\inf_{\nu:\, \nu^{\ms{up-dir}}[\tau_c](B_r(o))>b}h(\nu|\mu).$$ 
Here we used the definition $\nu^{\ms{up-dir}}[\tau_c](B_r(o))=\int_{B_r(o)}\one\{\SIR(x,o,\nu)< c\}\nu(\d x)$. 
Now we show that all minimizers preserve the rotational symmetry in the direct uplink scenario.

\begin{Proposition}
\label{uniqMinProp}
Let $\mu$ be a rotation-invariant intensity measure on $B_r(o)$ that has a strictly positive density with respect to the Lebesgue measure on $B_r(o)$.
In the direct-communication case, all minimizers in $J^{\ms{up-dir}}$ are rotationally invariant.
\end{Proposition}
\begin{proof}
Let $f(r)$ denote the radial density of $\mu$ with respect to the restriction of $\mathcal{H}_2$ on $B_r(o)$. Further, let $\nu\in \mathcal{M}(B_r(o))$ be absolutely continuous with respect to $\mu$, where the density will be denoted by $g$. Then, we define a new measure  $\nu'\in\mathcal{M}(B_r(o))$ whose density $g'$ w.r.t.~$\mu$ is given by 
$$g'(x)=g'(|x|)=\frac{1}{2\pi|x|}\int_{\partial B_{|x|}(o)}g(y)\mathcal{H}_1(\d y)$$
where $\mathcal{H}_1$ denotes the one-dimensional Hausdorff measure. 
Note that 
\begin{align}\label{CondMin}
\nu^{\ms{up-dir}}[\tau_c](B_r(o))=(\nu')^{\ms{up-dir}}[\tau_c](B_r(o)).
\end{align}
We claim that if $\nu$ is not rotation invariant, i.e., 
$\mathcal{H}_2(\{g\neq g'\})>0$, then
\begin{align}\label{CondEnt}
\int_{B_r(o)}g(x)f(|x|)\log g(x)\d x>\int_{B_r(o)}g'(x)f(|x|)\log g'(x)\d x.
\end{align}
But this, together with \eqref{CondMin}, would imply that $\nu$ can not be a minimizer of  $J^{\ms{up-dir}}$. 
In order to show \eqref{CondEnt}, let 
$$\mathcal{R}=\{0<s\le r:\; \text{ there exists }x\in \partial B_s(o) \text{ with }g(x)\ne g'(x)\}$$
 be the set of radii such that $g$ is not constant on $\partial B_s(o)$. Note that also 
$$\mathcal{H}_2(\{x\in B_r(o):\, |x|\in\mathcal{R}\})\ge \mathcal{H}_2(\{g\ne g'\})>0.$$
Then, by the coarea formula, which allows the disintegration into radii and angles (see~\cite{coarea}), it suffices to show that 
$$\frac{1}{2\pi s}\int_{\partial B_s(o)}g(x)\log g(x)\mathcal{H}_1(\d x)>g'(s)\log g'(s),$$
for all $s\in\mathcal{R}$, where $\mc{H}_1$ denotes the $1$-dimensional Hausdorff measure.
This is equivalent to 
$$\frac{1}{2g'(s)\pi s}\int_{\partial B_s(o)}g(x)\log g(x)\mathcal{H}_1(\d x)>\log g'(s),$$
for all $s\in\mathcal{R}$. Now, using the convexity of the function $-\log$, an application of Jensen's inequality shows that 
\begin{align*}
-\frac{1}{2g'(s)\pi s}\int_{\partial B_s(o)}g(x)\log \frac{1}{g(x)}\mathcal{H}_1(\d x)\ge\log g'(s),
\end{align*}
where equality occurs if and only if $g(x)$ is almost surely constant on $\partial B_s(o)$. 
\end{proof}

Next, we provide an approximate description of the minimizers in the direct uplink communication case. 
First note that the decay of the path-loss function implies that it is entropically more efficient to increase the interference at the origin by placing more users close to the base station. Second, if the interference at the origin is held fixed, then the SIR decays with the distance of the user to the base station. Hence, users with bad QoS will be located at the boundary rather than the center of the cell. The idea for the approximation is the following. We fix a radius $r\ge\r\ge0$ and compute the minimizer of the relative entropy under the constraint that a given SIR-threshold $c$ is met precisely at radius $\r$. In particular, this implies that in the region $\{x\in B_r(o):\, |x|> \r\}$ every user is disconnected. In order to achieve the desired proportion of disconnected users $b$ we use a flat profile in the outer annulus, since it is entropically favorable. The optimization has to be performed now over the radius $\r$ to balance the entropic costs of creating interference at the origin and increasing the number of disconnected users in the outer annulus.

Let $\mu$ be again the two-dimensional Lebesgue measure restricted to $B_r(o)$. Using variational calculus, as presented for example in \cite{gelfand}, we derive an expression for minimizers of $$\inf_{\nu:\, \SIR((\r,0),o,\nu)=c}h(\nu|\mu)$$ 
where $r\ge\r\ge0$ is a prescribed radius. The constraint under the infimum forces $\nu$ to have precisely $\ell(\r)/c$ interference at the origin. As we have seen before, we can assume that a minimizer $\nu$ has a radial symmetric density $f$ w.r.t.~$\mu$. Since $\nu$ is an extremal point of $h$ under the constraint 
\begin{align*}
\int_0^r2\pi sf(s)\ell(s)\d s=\frac{\ell(\r)}{c}
\end{align*}
using \cite[Section 12, Theorem 1]{gelfand}, there exists a constant $\a_\r$ such that the minimizing density has the form
$$f_\r(s)=\exp(\a_\r\ell(s)).$$
Using this density in the region $\{x\in B_r(o):\, |x|\le\r\}$ creates entropic costs of the form 
\begin{align*}
\gamma_{\text{int}}(\r)&=h(\nu{\mathord{\upharpoonright}_{B_\r(o)}}|\mu{\mathord{\upharpoonright}_{B_\r(o)}})=2\pi\int_0^\r sf_\r(s)[\log f_\r(s)-1]\d s+\r^2\pi\cr
&= 2\pi\int_0^\r se^{\a_\r\ell(s)}[\a_\r \ell(s)-1]\d s+\r^2\pi.
\end{align*}
In order to have $b$ users in the outer annulus using a flat profile, the density must be $b/(\pi(r^2-\r^2))$. The entropic costs of using this density in the outer annulus are given by 
\begin{align*}
\gamma_{\text{out}}(\r)&=h(\nu{\mathord{\upharpoonright}_{B_r(o)\setminus B_\r(o)}}|\mu{\mathord{\upharpoonright}_{B_r(o)\setminus B_\r(o)}})\cr
&=2\pi\int_\r^r\frac{sb}{\pi(r^2-\r^2)}[\log\frac{b}{\pi(r^2-\r^2)}-1]\d s+\pi(r^2-\r^2)\cr
&=b\log\frac{b}{\pi(r^2-\r^2)}-b+\pi(r^2-\r^2).
\end{align*}
Hence, 
we need to numerically compute $\a_\r$ from the equation $\int_0^r2\pi s e^{\a_\r\ell(s)}\ell(s)\d s=\frac{\ell(\r)}{c}$ and then optimize
\begin{align*}
\r\mapsto 2\pi\int_0^\r se^{\a_\r\ell(s)}[\a_\r\ell(s)-1]\d r-b\log[\pi(r^2-\r^2)].
\end{align*}
For the minimizing value $\r_{\text{min}}$ this leads to an approximate minimizing density of the form 
$$f(s)=e^{\a_{\r_{\text{min}}}\ell(s)}\one_{s\le \r_{\text{min}} }+\frac{b}{\pi(r^2-\r_{\text{min}}^2)}\one_{r\ge s> \r_{\text{min}}}.$$
In Figure~\ref{Radii}, we compare this density with the density observed by numerical simulations in a specific parameter set. This is the content of the next paragraph.
\begin{figure}[!htpb]
\centering
\includegraphics[width=0.4\textwidth]{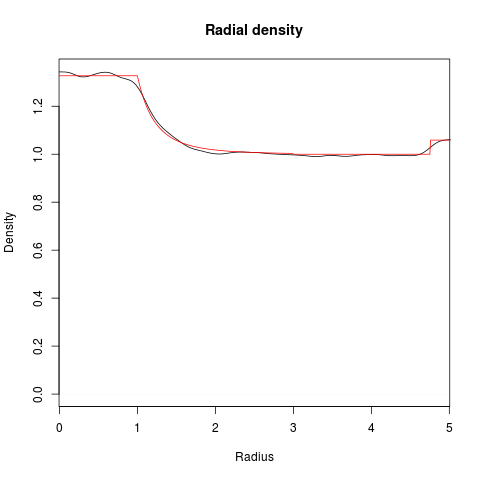}
\caption{Plot of the spatial intensity as a function of the radius in black. The corresponding approximation result is shown in red.}
\label{Radii}
\end{figure}

\medskip
In order to compare the asymptotic results, as $\lambda$ tends to infinity, with situations with finitely many users we present here some numerical simulations. Let us fix $r=5$, $\ell(s)=\min\{1,s^{-4}\}$, $b=0.1$, $\lambda=50$ and $c=c_0$
and consider the event $E_{50,c_0,0.1}^{\ms{up-dir}}$. In Figure~\ref{DirectSim} we present two realizations of the process where the left one is a typical configuration of users and the right one is a configuration which is an element of $E_{50,c_0,0.1}^{\ms{up-dir}}$. In accordance with Proposition \ref{uniqMinProp}, in the rare configuration, no breaking of the rotational symmetry can be observed. Under further inspection, a slightly higher intensity of users close to the origin can be detected. This higher intensity has the effect to create more interference around the origin leading to a \textit{screening} effect. In such a situation it is more likely for users to be unable to communicate with the base station. 

\begin{figure}[!htpb]
\centering
\includegraphics[width=0.4\textwidth]{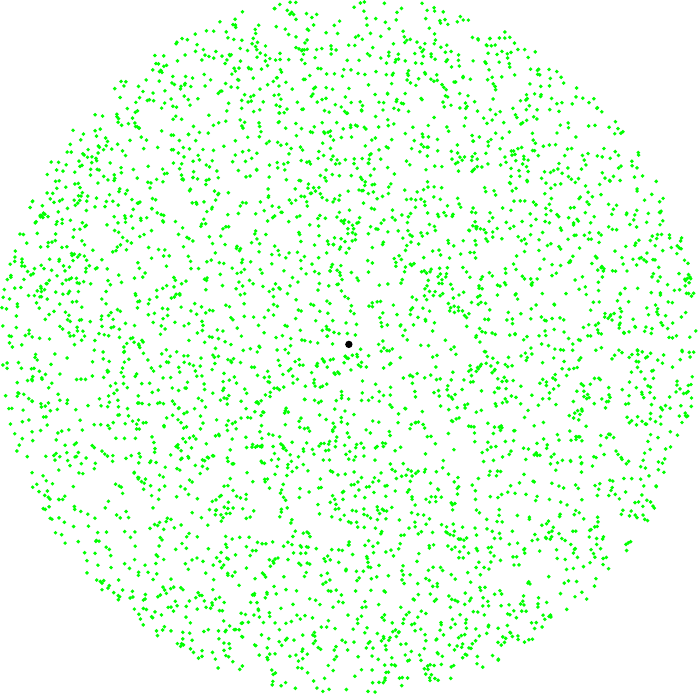}
\hspace{1.2cm}
\includegraphics[width=0.4\textwidth]{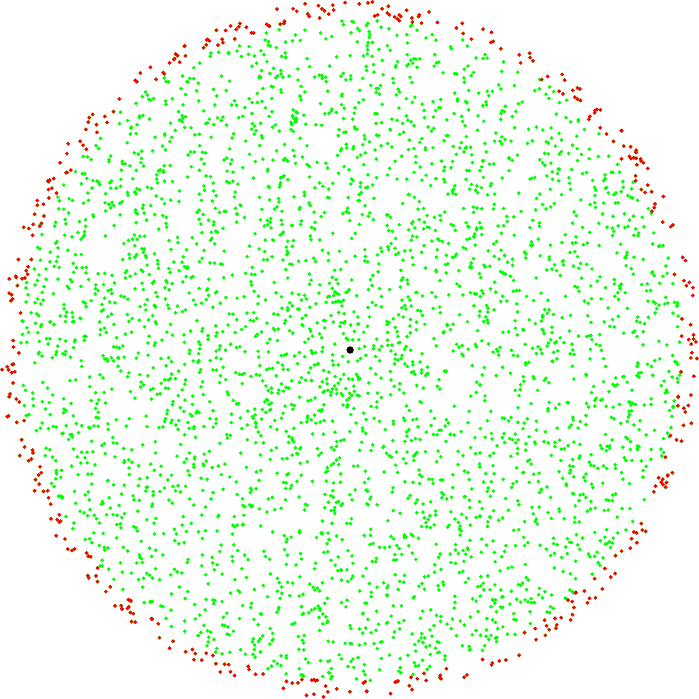}
\caption{Realizations of the direct communication network model. The typical realization on the left side shows no points unable to communicate with the base station at the origin. In the rare realization on the right side, the red users are unable to communicate with the base station.}\label{DirectSim}  
\end{figure}

Besides screening there is another possibility for the process to create an unexpectedly large number of users with bad connection quality. That is, to increase the number of users close to the boundary of the disk, since they become \textit{isolated} more easily. 

We have performed $N=10^8$ simulation runs. Using these simulations we obtained an estimate probability of the event $E_{50,c_0,0.1}^{\ms{up-dir}}$ given by approximately $1.8\times10^{-6}$. 
The black line in Figure~\ref{Radii} shows the radial intensity after performing a kernel-density estimate. In particular, we see that the intensity is substantially larger close to the origin, accounting for the screening effect, and the intensity is larger close to the cell boundary, accounting for isolation.

We want to mention that the plot of the density function in Figure~\ref{Radii} depends to a certain extent on the parameters of the kernel-density estimates. In particular, since there can not be any users below zero and above $5$, the kernel-density estimate tends to obscure the actual observations very close to the boundaries. In the plot we compensate for these effects by first mirroring users in our observations at the boundaries and then applying the kernel-density estimates. 
Another practical issue we want to address is the following. For the radial density plot, every observed user at radius $s$ must be given a weight proportional to $1/s$. This leads to a certain instability very close to the origin.

\subsection{Direct downlink communication}
Also in this case, we want to analyze configurations in the rare event, where an unexpected proportion of users experiences a bad QoS. More precisely we look at the event 
$$E_{\lambda,c,b}^{\ms{do-dir}}=\{L_{\lambda}^{\ms{do-dir}}[\tau_c](B_r(o))>b\},$$
which describes the existence of at least $\lambda\, b$ users in $B_r(o)$ that experience a connection quality that is less than $c$. 
In the direct downlink case, under a flat a priori intensity $\mu$, the expected minimum QoS is not necessarily attained at the boundary of the disk $B_r(o)$, but close to the boundary. This is due to the fact, that for users at the boundary, although the numerator in the SIR is minimum, there are fewer users in the vicinity, so that the expected interference for users away from the boundary is higher. Hence, the denominator of the SIR is not minimal for users at the boundary and the two competing effects balance each other at some radius $0<s_0<r$.

Set $r=5$ and $\ell(s)=\min\{1,s^{-4}\}$. The expected minimum QoS is close to the expected QoS $c_0$ at the boundary, which is much easier to compute
\begin{align*}
c_0&=\frac{5^{-4}}{\int_{B_5(o)} \min\{1,|x-5e_1|^{-4}\} \d x}
\approx 5.5\times 10^{-4}.
\end{align*}
In $E_{\lambda,c,b}^{\ms{do-dir}}$ we set $c=0.91c_0$ to compensate for the non-minimality of $c_0$ and for the fact, that in our simulations we have to work with a finite intensity $\lambda$, but $c_0$ is a limiting quantity as $\lambda$ tends to infinity. In particular for large enough $\lambda$ and $b>0$, $E_{\lambda,c,b}^{\ms{do-dir}}$ is a rare event.

In Figure~\ref{DirectDoSim} we present two realizations of the process where the left one is a typical configuration of users and the right one is a configuration which is an element of $E_{100,c,0.02}^{\ms{do-dir}}$. In the atypical configuration, the group of users with bad QoS forms a single localized cluster at the cell boundary. 
Here bad QoS is a consequence of a mutual screening effect due to too much interference at the user locations.
This indicates a breaking of rotational invariance in the set of minimizers. Heuristically multiple minimizers, ordered by the angle of the cluster of users with bad QoS, can be constructed in the following way. Let $X$ be a realization and $\s(X)$ the angle of the cluster centroid. Further let $f_{\a}:\R^2\to\R^2$ be the rotation by the angle $\a$, then $f_{-\s(X)+\a}(X)$ normalizes the angle of the cluster centroid of $X$ to $\a$. Now let $\nu$ be a minimizer, then the rotational invariance of the objective function implies that also $\nu_\a=\nu(f^{-1}_{-\s(\cdot)+\a}(\cdot))$ is a minimizer for all angles $\a$. But since $\nu_\a(\s)=\a$, we have constructed multiple minimizers indexed by $\a\in(0,2\pi]$.
\begin{figure}[!htpb]
\centering
\includegraphics[width=0.4\textwidth]{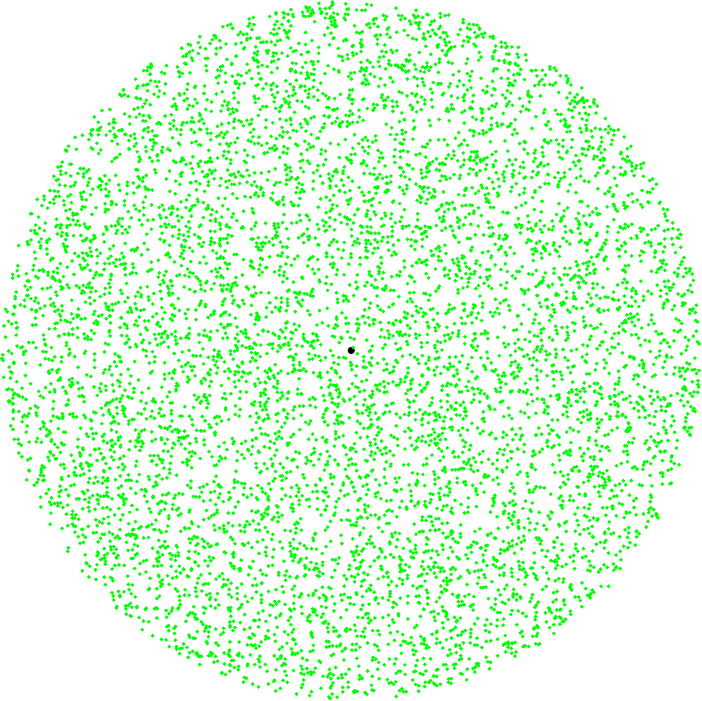}
\hspace{1.2cm}
\includegraphics[width=0.4\textwidth]{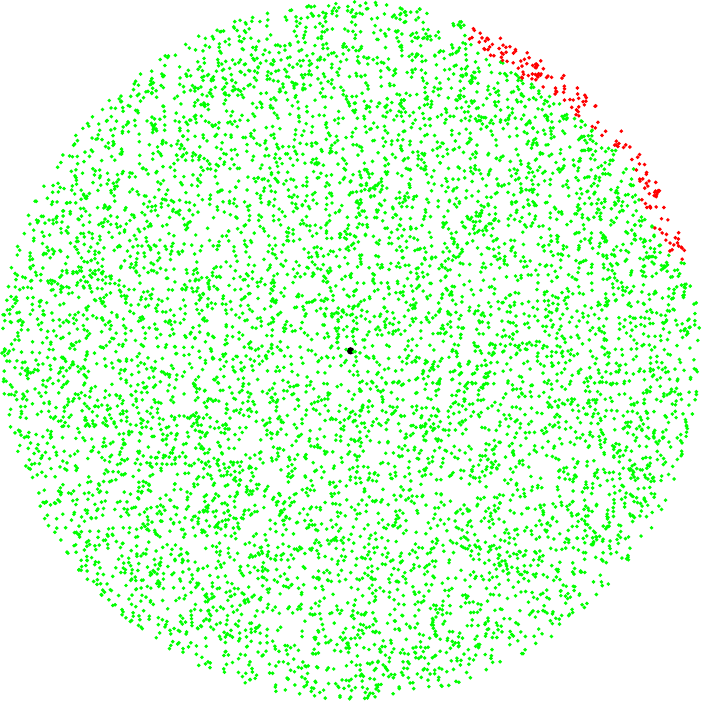}
\caption{Realizations of the direct downlink communication network model. All users and the base station contribute to the interference. Only green users can be reached directly from the base station.}
\label{DirectDoSim} 
\end{figure}

\subsection{Relayed communication}
As illustrated by our simulations discussed above, in direct-communication networks the event of observing a large number of frustrated users who cannot attain a desired QoS is often caused by a large number of users close to the cell boundary. 
Indeed, the path-loss of those users is so pronounced that even a small increase of the interference at the target can lead to the event of not achieving the desired QoS. In case of relayed communication there can be similar effects observed as in direct communication, but new phenomena also arise. Still the effect are similar in case of relayed up- and relayed downlink, so we will focus only on the uplink case.
As in the case of direct communication, a way to create configurations in the frustration event
$$E_{\lambda,c,b}^{\ms{up}}=\{L_{\lambda}[\tau_c](B_r(o))>b\},$$
where the parameters $c$ and $b$ are chosen appropriately, is to screen the origin or to increase isolated users far away from the origin. 

In contrast to the direct-communication case, in relayed communications, bad QoS can be a consequence of the absence of relays. This can be seen in the following setup using simulations. Assume the a priori density $\mu$ to be rotational invariant with high density in a small circle around the origin and close to the boundary of the disk. Additionally, assume that $\mu$ is zero everywhere else except for a small strip at approximately half the radius of the disk, here it is positive but low. Users close to the boundary are too far away from the base station to establish direct communication but can connect via users in the strip serving as relays. The left image in Figure~\ref{SymBreak} shows a typical realization according to $\lambda\mu$. In order to see a configuration in $E_{\lambda,c,b}^{\ms{up}}$ it is entropically not very costly to avoid placing users in the strip, and even cheaper to do that in a small section of the strip only. Such a realization is shown on the right side of Figure~\ref{SymBreak}.

\begin{figure}[!htpb]
               \centering
\includegraphics[width=0.4\textwidth]{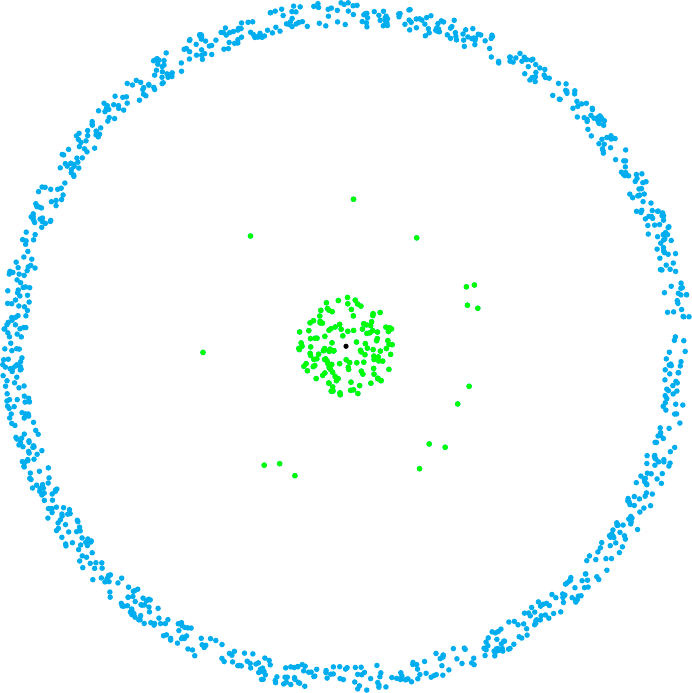}
\hspace{1.2cm}
\includegraphics[width=0.4\textwidth]{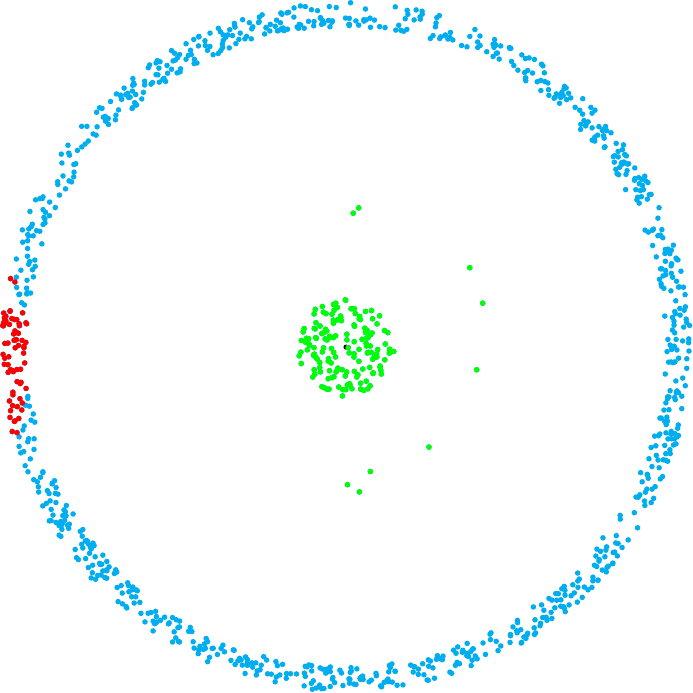}
\caption{Realizations of the relayed communication network model for an inhomogeneous intensity $\mu$. The typical realization on the left side shows green points, able to communicate with the base station at the origin and blue points, able to communicate with the base station using green points as relays. 
In the rare realization on the right side, the asymmetrically distributed red users are unable to communicate with the base station due to missing relays in the strip.}
\label{SymBreak}
\end{figure}

\section*{Acknowledgments}
This research was supported by the Leibniz program Probabilistic Methods for Mobile Ad-Hoc Networks. The authors thank W.~K\"onig for interesting discussions and comments.

\bibliography{../wias}

\begin{thebibliography}{10}

\bibitem{3GPPRelay}
3GPP.
\newblock Relay architectures for {E-UTRA (LTE-Advanced)}.
\newblock {\em Evolved Universal Terrestrial Radio Access (E-UTRA)}, TR 36.806,
  2010.

\bibitem{sprinkling}
M.~Aizenman, J.~T. Chayes, L.~Chayes, J.~Fr{\"o}hlich, and L.~Russo.
\newblock On a sharp transition from area law to perimeter law in a system of
  random surfaces.
\newblock {\em Comm. Math. Phys.}, 92(1):19--69, 1983.

\bibitem{baccelli2009stochastic1}
F.~Baccelli and B.~B{\l}aszczyszyn.
\newblock {\em Stochastic Geometry and Wireless Networks: Volume 1: Theory}.
\newblock Now Publishers Inc, 2009.

\bibitem{baccelli2009stochastic2}
F.~Baccelli and B.~B{\l}aszczyszyn.
\newblock {\em Stochastic Geometry and Wireless Networks: Volume 2:
  Application}.
\newblock Now Publishers Inc, 2009.

\bibitem{Keeler15}
B.~B{\l}aszczyszyn and H.~P. Keeler.
\newblock Studying the {SINR} process of the typical user in poisson networks
  by using its factorial moment measures.
\newblock {\em IEEE Trans. Inform. Theory}, 2016, to appear.

\bibitem{bletsas1}
A.~Bletsas, A.~Khisti, D.~P. Reed, and A.~Lippman.
\newblock A simple cooperative diversity method based on network path
  selection.
\newblock {\em IEEE J. Sel. Areas Commun.}, 24(3):659--672, 2006.

\bibitem{bletsas2}
A.~Bletsas, H.~Shin, and M.~Z. Win.
\newblock Cooperative communications with outage-optimal opportunistic
  relaying.
\newblock {\em IEEE Trans. Wireless Commun.}, 6(9):3450--3460, 2007.

\bibitem{relAss3}
J.~Cao, T.~Zhang, Z.~Zeng, and D.~Liu.
\newblock Interference-aware multi-user relay selection scheme in cooperative
  relay networks.
\newblock In {\em Globecom Workshops, 2013 IEEE}, pages 368--373, 2013.

\bibitem{decrRel}
Y.~Chen, P.~Martins, L.~Decreusefond, Y.~Feng, and X.~Lagrange.
\newblock Stochastic analysis of a cellular network with mobile relays.
\newblock In {\em Global Communications Conference (GLOBECOM), 2014 IEEE},
  pages 4758--4763, 2014.

\bibitem{dz98}
A.~Dembo and O.~Zeitouni.
\newblock {\em Large Deviations Techniques and Applications}.
\newblock Springer, New York, second edition, 1998.

\bibitem{adHoc}
H.~D\"oring, G.~Faraud, and W.~K\"onig.
\newblock Connection times in large ad-hoc mobile networks.
\newblock {\em Bernoulli}, 2016, to appear.

\bibitem{sinrPerc}
O.~Dousse, M.~Franceschetti, N.~Macris, R.~Meester, and P.~Thiran.
\newblock Percolation in the signal to interference ratio graph.
\newblock {\em J. Appl. Probab.}, 43(2):552--562, 2006.

\bibitem{Dudley}
R.~M. Dudley.
\newblock {\em Real Analysis and Probability}.
\newblock Cambridge University Press, Cambridge, 2002.

\bibitem{coarea}
H.~Federer.
\newblock {\em Geometric Measure Theory}.
\newblock Springer, New York, 1969.

\bibitem{ldpInt}
A.~J. Ganesh and G.~L. Torrisi.
\newblock Large deviations of the interference in a wireless communication
  model.
\newblock {\em IEEE Trans. Inform. Theory}, 54(8):3505--3517, 2008.

\bibitem{gelfand}
I.~M. Gelfand and S.~V. Fomin.
\newblock {\em Calculus of Variations}.
\newblock Prentice-Hall, Englewood Cliffs, 1963.

\bibitem{Kingman93}
J.~F.~C. Kingman.
\newblock {\em Poisson Processes}.
\newblock Oxford University Press, New York, 1993.

\bibitem{relAss1}
I.~Krikidis, J.~S. Thompson, S.~McLaughlin, and N.~Goertz.
\newblock Max-min relay selection for legacy amplify-and-forward systems with
  interference.
\newblock {\em IEEE Trans. Wireless Commun.}, 8(6):3016--3027, 2009.

\bibitem{m2mHop}
D.~Malak, H.~S. Dhillon, and J.~G. Andrews.
\newblock Optimizing data aggregation for uplink machine-to-machine
  communication networks.
\newblock {\em arXiv preprint arXiv:1505.00810}, 2015.

\bibitem{fogComp}
M.~Peng, S.~Yan, K.~Zhang, and C.~Wang.
\newblock Fog computing based radio access networks: Issues and challenges.
\newblock {\em arXiv preprint arXiv:1506.04233}, 2015.

\bibitem{relAss2}
J.~Si, Z.~Li, and Z.~Liu.
\newblock Outage probability of opportunistic relaying in {R}ayleigh fading
  channels with multiple interferers.
\newblock {\em IEEE Signal Process. Lett.}, 17(5):445--448, 2010.

\bibitem{isSir}
G.~L. Torrisi and E.~Leonardi.
\newblock Simulating the tail of the interference in a {P}oisson network model.
\newblock {\em IEEE Trans. Inform. Theory}, 59(3):1773--1787, 2013.

\bibitem{vazIyer}
R.~Vaze and S.~Iyer.
\newblock Percolation on the information-theoretically secure signal to
  interference ratio graph.
\newblock {\em J. Appl. Probab.}, 51(4):910--920, 2014.

\bibitem{weberSIC}
S.~P. Weber, X.~Yang, J.~G. Andrews, and G.~de~Veciana.
\newblock Transmission capacity of wireless ad hoc networks with outage
  constraints.
\newblock {\em IEEE Trans. Inform. Theory}, 51(12):4091--4102, 2005.

\bibitem{fullDuplex1}
Z.~Zhang, X.~Chai, K.~Long, A.~V. Vasilakos, and L.~Hanzo.
\newblock Full duplex techniques for {5G} networks: self-interference
  cancellation, protocol design, and relay selection.
\newblock {\em IEEE Commun. Mag.}, 53(5):128--137, 2015.

\end{thebibliography}
\bibliographystyle{abbrv}

\end{document}